\documentclass[12pt,reqno]{amsart}
\usepackage{amssymb}
\usepackage[all]{xy}

\def\Q {{\mathbb Q}}
\def\A{{\mathbb A}}

\def\cL {{ \mathcal L}}
\def\cE {{ \mathcal E}}
\def\R {{\mathbb R}}
\def\N {{\mathbb N}}

\def\cF {{\mathcal{ F}}}

\def\Z {{\mathbb Z}}

\def\C {{\mathbb C}}

\def \Im{\mbox{Im\,}}

\newtheorem{theorem}{Theorem}

\newtheorem{proposition}{Proposition}

\newtheorem{lemma}{Lemma}

\newtheorem{definition}{Definition}

\newtheorem{comment}{Comment}

\theoremstyle{remark}
\def \b {\bigskip}
\def \m {\medskip}
\begin{document}

\title []{On the analogy between  $L-$functions
and  Atiyah-Bott-Lefschetz trace formulas for foliated spaces}

\author{Eric Leichtnam}
\date{\today}

\address{Institut Math\'ematique de Jussieu-PRG et CNRS, Batiment Sophie Germain
 75013 Paris,  France}

\begin{abstract}
Christopher Deninger has developed an infinite dimensional
cohomological formalism which would allow to prove 
the expected properties of the motivic $L-$functions (including the
 Dirichlet $L-$ functions).
These cohomologies are (in general) not yet constructed.
Deninger has argued that they  might be 
constructed as  leafwise cohomologies associated to ramified leafwise flat 
vector bundles on suitable
foliated spaces. In the case of number fields we propose a set of axioms allowing to make this more precise and to motivate new theorems. We also check the coherency of these axioms and from 
them we derive "formally" an Atiyah-Bott-Lefschetz trace formula which would imply Artin conjecture for a Galois extension of $\Q$.
\end{abstract}

\maketitle

{\it Tribute to Gianni Rivera for his 70th birthday.}
\tableofcontents

\section{Introduction.}$\;$

Christopher Deninger 's 
approach to the study of arithmetic zeta and motivic $L-$functions proceeds in two steps (see for instance \cite{D 3}, \cite{D 3b} ).

In the {\bf  first step}, he postulates the existence of infinite dimensional 
cohomology groups satisfying some "natural properties". From these data, he has elaborated a formalism which would  allow him to prove 
 the expected 
properties for the arithmetic zeta functions: functional equation,
conjectures of Artin, Beilinson, Riemann...etc. There it is crucial 
to interpret the so called explicit formulae for 
the zeta and motivic $L-$functions as  Lefschetz trace formulae. 

The {\bf second step} consists in
 constructing  these cohomologies. 
Deninger has given some hope that for zeta functions these cohomologies might be constructed 
as leafwise cohomologies of suitable foliated spaces. Moreover, in the case of motivic $L-$functions, one should consider 
flat (ramified) vector bundles on the corresponding foliated space, see Deninger \cite{D 0c}.
Very little is known in this direction
at the moment, but this second step seems to be a good motivation 
to develop interesting mathematics even if they remain far from
the ultimate goal.

In Section 2 we recall Deninger's cohomological formalism
in the case of the Riemann zeta and (primitive) Dirichlet $L-$functions $\Lambda (\chi , s)$. We point out 
a dissymmetry in the explicit formulae \eqref{eq:EF}  and \eqref{eq:EFchi} between the 
coefficients of $\delta_{ k \log p}$ and  $\delta_{ -k \log p}$, see 
Comment \ref{com:disym}.

In Section 3 is devoted to the description of the  Lefschetz trace formula for 
a flow acting on a codimension one foliated space. 
In Section 3.1 we recall the Guillemin-Sternberg trace formula which 
is indeed an important computational tool for this goal.

In Section 3.2 we recall the theorem of Alvarez-Lopez and Kordyukov.

They consider a flow  $(\phi^t)_{t \in \R}$ acting on $(X, \mathcal{F})$ where 
the compact three dimensonal manifold $X$ is foliated by Riemann surfaces.
They assume that $(\phi^t)_{t \in \R}$ preserves globally the foliation and 
is transverse to the foliation. Then Alvarez-Lopez and Kordyukov 
define a suitable leafwise Hodge cohomology on which $\phi^t$ acts and they prove an  Atiyah-Bott-Lefschetz trace 
 formula (Theorem \ref{thm:AK})
which has some similarities with $\eqref{eq:EF}$ for $t$ real positive. The  dissymmetry mentioned 
above for (1) does not hold here. Nevertheless, by comparison with \eqref{eq:EF},
 it suggests that there should exist a flow $(\phi^t)_{t \in \R}$ acting on a certain 
 space $S_\Q$ with the following property.
To each prime number $p$ [resp. the archimedean place of $\Q$] there 
should correspond a closed orbit with length $\log p$ [resp. a stationary point] 
of the flow $\phi^t.$ But since the   dissymmetry mentioned 
above for (1) does not hold in Theorem \ref{thm:AK}, the space $S_\Q$ cannot be a foliated manifold but rather a so called laminated foliated space. Its transverse structure might involve 
the $p-$adic integers or even the adeles (see \cite{D 2} and  \cite{El1}   for a simple example), but we shall not dwell here on this important point.

In Section 3.3 we introduce  the concept of a ramified flat line bundle $\cL_\rho \rightarrow X $ around a finite number of closed orbits, and 
define the associated leafwise Hodge cohomology groups $H^j_\tau ( X ; \cL_\tau)$, $0\leq j \leq 2$. We then prove 
a Lefschetz trace formula where the ramified closed orbits of  $(X, \phi^t)$ do not appear: see  Theorem \ref{thm:traceram}. This trace formula has some similarities with the explicit formula 
\eqref{eq:EFchi} for Dirichlet $L-$functions, the ramified closed orbits corresponding to the (ramified) prime numbers which divide the conductor 
of the Dirichlet character.

In Section 4, we consider more generally the Dedekind zeta function $\widehat{\zeta}_K(s)$ of a number field $K$ and recall 
the associated explicit formula \eqref{eq:EFK}. Then, making  a synthesis of  Deninger's work, 
 we state several assumptions for a laminated foliated 
 space $(S_K, \mathcal{F}, g, \phi^t)$ which (if satisfied) would allow to construct the required
 (leafwise) cohomology groups for the Dedekind zeta function. In particular, the explicit formula \eqref{eq:EFK} should be interpreted as 
 a (leafwise) Atiyah-Bott-Lefschetz trace formula. 
 We compare carefully the contributions
of the archimedean places of $K$ in \eqref{eq:EFK} with the contribution 
of a stationary point in the Guillemin-Sternberg formula: we explain an apparent incompatibility in the case of real places.
Next, we come back to the case of a primitive Dirichlet character  $\chi $  mod $m$  and consider the cyclotomic field $K= \Q[e^{\frac{ 2 i \pi } {m}}]$ associated to $\chi$ with Galois group $G= ( \Z / m\Z)^*$. So $\chi$ defines a group homomorphism $\chi: G \rightarrow S^1$.
We consider the ramified flat complex line bundle over $S_\Q$:
$$
\cL_\chi= \frac{ S_K \times \C}{G} \rightarrow S_\Q\,.
$$ and define leafwise cohomology groups $ \overline{H}^j(\cL_\chi)$, $0 \leq j \leq 2$. 

Then, imitating the proof of  our Theorem \ref{thm:traceram} and using the assumptions of Section 4,  we interpret  "formally"  the explicit formula \eqref{eq:EFchi} for the Dirichlet $L-$function $\Lambda (\chi , s)$
 as a leafwise Lefschetz trace formula for the vectors spaces  $ \overline{H}^j(\cL_\chi)$, $0 \leq j \leq 2$ :
see Theorem \ref{thm:Dir}. We insist on the fact that it is not known whether or not the assumptions of Section 4 are satisfied.

In Section 5, we assume that $K$ is a (finite) Galois extension of $\Q$ with Galois group $G$. We review the definition and properties of the Artin $L-$function $ \Lambda (K, \chi, s)$ associated to 
an irreducible representation $\rho: G \rightarrow GL_\C(V)$. We recall the standard explicit formula \eqref{eq:standard} for $ \Lambda (K, \chi, s)$: its spectral 
side involves the zeroes (with sign $-$) and the poles (with sign $+$). These poles are not controlled 
because the $\Gamma-$factor introduced in the definition \eqref{def:A} of $ \Lambda (K, \chi, s)$ is not 
associated to a mathematical structure, this $\Gamma-$factor seems to come as a parachute. This situation 
is in sharp contrast with the case of the Dirichlet $L-$functions recalled in Section 2: there the $\Gamma-$factor 
is introduced in the Dirichlet $L-$function in order to express it as the Mellin transform of a suitable theta function.

 Next we consider the ramified flat vector bundle over $S_\Q$:
$$
\cE_\rho= \frac{ S_K \times V}{G} \rightarrow S_\Q\,,
$$ and define leafwise cohomology groups $ \overline{H}^j(\cE_\rho)$, $0 \leq j \leq 2$.

Then, imitating the proof of  our Theorem \ref{thm:traceram} and using the axioms of Section 4,  we "prove  formally" a leafwise Atiyah-Bott-Lefschetz trace formula for the vectors spaces $ \overline{H}^j(\cE_\rho)$ ($0 \leq j \leq 2$) which provides an explicit formula
 \eqref{eq:Artin}. In this formula, the $\Gamma-$factor  of $ \Lambda (K, \chi, s)$ {\it appears naturally} in the computation 
 of the contribution of the fixed points.
 In the spectral side of \eqref{eq:Artin}, the numbers only appear with a sign $-$. Since the geometric sides of  \eqref{eq:Artin} and 
 \eqref{eq:standard} coincide, one then would get formally that the Artin $L-$function has no poles !!

Thus, the $ \overline{H}^j(\cE_\rho)$ ($0 \leq j \leq 2$) seem "to provide" a construction of  the cohomology groups 
that Deninger's cohomological formalism attributes to $ \Lambda (K, \chi, s)$.
In \cite{D 3b}[Section 3], working in the general cohomological formalism that he elaborated and using regularized determinants, Deninger has reduced the validity of Artin conjecture for simple motives to the vanishing 
of $H^0$ and $H^2$. Our approach via the trace formula is a bit different 
and possibly  simpler in the sense that it seems to need  less foundational results.

It would be interesting to confront certain ideas from automorphic theory (see eg \cite{Gelbart}, \cite{Lapid}, \cite{Taylor}) 
with the axioms of Sections 4.3. Indeed, one of the goals of Langlands programme is to identify
$ \Lambda (K, \chi, s)$ with the $L-$function $L_\pi$ of an automorphic cuspidal representation $\pi$. In the automorphic world, 
the $\Gamma-$factor of $L_\pi$ {\it appears naturally} in  a mathematical structure. Therefore, one may ask if the axioms 
of Section 4.3 and the data of $\rho: G \rightarrow GL_\C (V)$ could allow to construct " formally " the desired  automorphic 
cuspidal representation $\pi$ associated to $ \Lambda (K, \chi, s)$.

Ralf Meyer \cite{Meyer} has provided a nice and new spectral interpretation of the explicit formula \eqref{eq:EFK} (actually Meyer considers all 
Hecke $L-$functions at the same time). Unfortunately, the action of $G$ on Meyer's cohomology groups is trivial. The fact that this action is not trivial 
in our (conjectural) setting is guaranted by the axioms of Section 4.3.
\b

In another direction, it should also be interesting to confront the properties of the hypothetic  foliated space 
 $(S_K, \mathcal{F}, g , \phi^t)$ 
with ideas from Topos theory. For instance see Morin (\cite{Morin} and \cite{FM}),  Caramello \cite{Cara}, written talks by Laurent Lafforgue \cite{Lafforgue} and, Connes's lecture 2013 in 
Coll\`ege de France. In any case, we hope that this paper will be useful to somebody else.

\m
\noindent  {\bf Acknowledgements.} I am happy to thank David Harari and Georges Skandalis for having pointed out corrections, and also Vincent Lafforgue for a useful trick.
I am also happy to thank Christopher Deninger, Jesus Alvarez-Lopez and Bora Yalkinoglu for helpful comments.

\section{Deninger's Cohomological formalism in the case of the Dirichlet $L-$ functions.}  
\label{sect:zeta}$\;$

\subsection{Dirichlet $L-$ functions $\Lambda(\chi,s)$.} $\;$

\m

The (completed) Riemann zeta function is given by:
$$
\widehat{\zeta}(s) = 
 \pi^{-s/2} \Gamma(s/2) \, \prod_{p \in
{\mathbb{P}}} \,\frac{1}{1 - p^{-s}}
$$ 
where ${\mathbb{P}}=\{2,3,5,\ldots\}$ denotes the usual set of prime
numbers.  The following well known explicit formulas express a
connection between ${\mathbb{P}} \cup\{\infty\}$ and the zeroes of
$\widehat{\zeta}$. Let $\alpha \in C^\infty_{compact}(\R, \R)$ and for 
$s\in \R$, set $\Phi(s)= \int_{\R} e^{s t} \alpha(t) dt$; $\Phi$ belongs 
to the Schwartz class $\mathcal{S}(\R).$  Then one can prove
the following formula:
\begin{equation} 
\label{eq:EF}
\begin{aligned}[rcl]
\Phi(0)  & - \sum_{\rho \in \widehat{\zeta}^{-1}\{0\} ,\, \Re \rho \geq 0} \Phi(\rho) +
\Phi(1) = \\ & =\sum_{p \in \mathbb{P}} \log p \left( \sum_{k \geq 1}
\alpha(k \log p) \,+\, \sum_{k \leq -1} p^k \alpha(k \log p) \right) +
W_\infty(\alpha),
\end{aligned}
\end{equation}
where 
\[
W_\infty(\alpha) = \alpha(0) \log \pi + \int_0^{+\infty} \left (
\frac{\alpha(t) + e^{-t} \alpha(-t)} {1- e^{-2 t}} -\alpha(0) \frac
{e^{-2 t}} {t} \, \right) dt.
\]

Let $m\in \N \cap [3 , + \infty[$, a  Dirichlet character $\chi$ mod $m$ is a group homomorphism 
$$\chi : (\Z/ m \Z)^* \rightarrow S^1\, .
$$ Such a character is called primitive if there exists no non trivial divisor $m^\prime$ of $m$
such that  $\chi=  \chi^\prime \circ \pi$ where  $\chi^\prime$ is a Dirichlet character  mod  $m^\prime$ 
and $\pi: (\Z/ m \Z)^* \rightarrow (\Z/ m^\prime \Z)^*$ denotes the projection. The great commun divisor 
$f$ of all such divisors $m^\prime$ is called the conductor of $\chi$, one can check that 
$\chi$ is induced by a primitive Dirichlet character mod $f$.

A Dirichlet character $\chi $ mod $m$ induces a multiplicative map, still denoted $\chi$,
from $\Z$ to $S^1 \cup \{0\}$ by the rules:
$$
\forall n \in \Z,\; \chi (n)= \chi (n + m \Z) \; {\rm if}\, n \wedge m = 1,\; \chi (n) =0 \; {\rm if}\, n \wedge m \not= 1\,, \,\chi(0)=0\,.
$$

A Dirichlet character $\chi $ mod $m$ is primitive if and only if  for any non trivial divisor $m^\prime$ of $m$,
\begin{equation} \label{eq:prim}
\exists a \in \Z, \, a \wedge m =1,\, a = 1 \, {\rm mod}\, m^\prime ,\, \chi(a) \not= 1\,.
\end{equation}

Consider a (non trivial) primitive Dirichlet character $\chi$ mod $m$, define $q \in \{0,1\}$ 
by $\chi(-1) = (-1)^q$. Then the following function, first defined on the half plane 
$\Re s >1$,  extends as an entire holomorphic function on $\C$:
$$
\Lambda (\chi, s) =( \frac{m}{\pi})^{s/2} \Gamma ( \frac{s+q}{2}) \prod_{p \in {\mathbb{P}},\, p\wedge m= 1}\frac{1}{1-p^{-s}}\,.
$$

It satisfies the functional equation:
\begin{equation} \label{eq:eqf}
\forall s \in \C,\; \Lambda (\chi, s)= \frac{\sum_{a=0}^{m-1} \chi(a) e^{\frac{2 i \pi a}{m} }}{i^q \sqrt{m}} \Lambda (\overline{\chi}, 1-s)\,.
\end{equation}The proof uses  first the fact that $\Lambda (\chi, s)$ is the Mellin transform at $\frac{s+q}{2}$ of the theta function
$$
\theta(\chi , y) = \frac{1}{2}\, (\frac{\pi}{m})^{q/2}\,\sum_{n\in \Z} \chi(n) n^q e^{\frac{-n^2 \pi y}{m}}
$$ and then a certain relation between $\theta(\chi , 1/y)$ and $\theta(\overline{\chi} , y)$ which is established with the help of  the Gauss sums $\sum_{a=0}^{m-1} \chi(a) e^{\frac{2 i \pi a n}{m} }$ ($n \in \Z$).

Let $\alpha \in C^\infty_{compact}(\R, \R)$ such that $\alpha(0)=0$ and for 
$s \in \R$, set $\Phi(s)= \int_{\R} e^{s t} \alpha(t) dt .$ One then has:
\begin{equation} \label {eq:EFchi}
\begin{aligned}[rcl]
-\sum_{\rho \in \Lambda(\chi, \cdot)^{-1} \{0\}} \Phi(\rho)= &
\sum_{p \in \mathbb{P}, \, p\wedge m =1} \log p \sum_{n\geq 1} \bigl( \chi(p)^n \alpha( n\log p) + 
p^{-n} \chi(p)^{-n} \alpha(-n  \log  p) \bigr)  \\&
+ \int_{0}^{+\infty} \frac{\alpha (x) e^{-q x} + \alpha(-x) e^{- x (1+q)}}{1- e^{- 2 x}}\, d x\,.
\end{aligned}
\end{equation}

By comparison with \eqref{eq:EF}, "$\Phi(0) + \Phi(1)$ has disappeared", which means that $ \Lambda(\chi, \cdot)$ has no poles.

\bigskip
The idea of the proof of 
\eqref{eq:EFchi} is the following: apply the residue theorem to the integral of the function
$$
s\mapsto \left(\int_0^{+\infty} \alpha(\log t) \,t^s \frac{d\,t}{t}\right)\, \frac { \Lambda^\prime (\chi, s) } {  \Lambda (\chi, s)}$$
along the boundary of the rectangle of $\C$ defined by the four  points: 
$$
1+\epsilon +i T,\, -\epsilon +i T ,\,
-\epsilon -i T,\,1+\epsilon -i T,
$$ then use the functional equation  \eqref{eq:eqf}
and the formula:
$$
\frac {\Gamma^\prime } { \Gamma}(\frac {s}{2})=
\int_0^{+\infty} \left( \frac{e^{-u} }{u} - \frac { e^{-u\frac{s}{2} } }{1-e^{-u}}\right) du,
$$ lastly let $T$ goes to $+\infty$.


\b
\subsection{ Deninger's cohomological formalism.} $\;$

\b
Deninger's philosophy is motivated by the fact that
the left hand side of \eqref{eq:EF}
$$ \Phi(0)   - \sum_{\rho \in \widehat{\zeta}^{-1}\{0\},\, \Re \rho \geq 0} \Phi(\rho) +
\Phi(1)
$$ is reminiscent of a Lefschetz trace formula of the form
$$
{\rm TR}\, \int_\R \alpha(t) e^{t\, \Theta_0}\,d t - {\rm TR}\, \int_\R \alpha(t) e^{t\, \Theta_1}\, d t  + {\rm TR}\, \int_\R \alpha(t) e^{t\, \Theta_2}\, d t,
$$ where the following  two assumptions should be satisfied.

$\bullet$ 
$\Theta_0=0$ acts on $H^0=\R$, $\Theta_2={\rm Id}$ acts on $H^2=\R$.

$\bullet$$\bullet$ The closed unbounded operator,
$\Theta_1$ acts on an infinite dimensional real vector (pre-Hilbert) space $H^1$ and has discrete spectrum. 
For any  $ \alpha \in C^\infty_{compact} (\R , \R)$ the operator $ \int_\R \alpha(t) e^{t\, \Theta_1}\, d t $  is trace class. The 
eigenvalues of $\Theta_1 \otimes Id_\C$  acting on $H^1\otimes_\R \C$ coincide with the non trivial zeroes of  $\widehat{\zeta}$.

\medskip 
For each primitive (non trivial) Dirichlet character $\chi$, there should exist an infinite dimensional complex (pre-Hilbert) 
space $H_\chi^1$ endowed with a closed unbounded operator $\Theta_{1,\chi} : H_\chi^1 \rightarrow H_\chi^1$ 
such that 
 for any  $ \alpha \in C^\infty_{compact} (\R , \R)$ the operator $ \int_\R \alpha(t) e^{t\, \Theta_{1,\chi}}\, d t $  is trace class.
Moreover, the eigenvalues of $\Theta_{1,\chi}$ coincide with the non trivial zeroes of $L_\chi$.

In Deninger's approach one first assumes the existence of a Poincar\'e duality 
pairing: 
$$
H_\chi^1 \times H_{\overline{\chi}}^1 \rightarrow H^2
$$

$$
(\alpha, \beta) \rightarrow \alpha \cup \beta
$$ satisfying

\begin{equation} \label{eq:Poincare}
\forall ( \alpha, \beta) \in H_\chi^1\times H_{\overline{\chi}}^1,\; e^{t\, \Theta_{1,\chi}}\alpha \cup e^{t\, \Theta_{1,\overline{\chi}}}\beta= 
e^t (\alpha \cup \beta ),
\end{equation} where the $e^t$ is dictated by the fact that  $\Theta_2= Id$ on $H^2$.

In order to (re)prove the functional equation \eqref{eq:eqf}, one combines the property \eqref{eq:Poincare} and 
the following identity (which Deninger assumes to hold true):
$$
\Lambda (\chi, s) = C \, \frac{ det_\infty ( \frac{ s - \theta_{1, \chi}}{2 \pi}\,: H^1_\chi )}{ det_\infty ( \frac{ s - \theta_{1, \chi}}{2 \pi}\,: H^0_\chi )\, \, 
det_\infty ( \frac{ s - \theta_{1, \chi}}{2 \pi}\,: H^2_\chi ) }\; ,$$
where $C$ is a constant depending on a choice of conventions. (Actually in this particular case, the denominator is identically equal to $1$).

Second, one assumes the existence of  an anti-linear Hodge star $\star$ operator:
$$
\star\,: H_\chi^1 \rightarrow H_{\overline{\chi}}^1 ,\; \star\,: H_{\overline{\chi}}^1 \rightarrow H_\chi^1
$$ such that $\star^2 =Id$ and 
$ \star e^{t\, \Theta_{1, \chi}}  =  e^{t\, \Theta_{1, \overline{\chi}}} \star$
 and $< \alpha ; \beta >= \alpha \cup \star \beta$ defines a  scalar product (anti-linear on the right)
 on the  vector space $H_\chi^1.$ 
 
 These data imply easily the following:
 \begin{equation} \label{eq:hilb}
 \forall \alpha , \alpha^\prime \in H_\chi^1,\;
 < e^{t\, \Theta_{1, \chi}}\alpha ; e^{t\, \Theta_{1, \chi}}\alpha^\prime > =e^t <\alpha ; \alpha^\prime >.
\end{equation}
Therefore, 
$$\frac { d}{d t} < e^{t\, \Theta_{1,\chi}}\alpha ; e^{t\, \Theta_{1,\chi}}\alpha >_{t=0}
=<  \Theta_{1,\chi}( \alpha) ; \alpha > +
 <   \alpha ;  \Theta_{1, \chi} (\alpha) >=<\alpha ; \alpha >,
$$ 
and
$$
<  (\Theta_{1,\chi}-1/2)( \alpha) ; \alpha > +
 <   \alpha ; ( \Theta_{1,\chi}-1/2)( \alpha) >=0.
 $$
 Therefore, the
eigenvalues $s$
of $\Theta_{1, \chi}$ (which coincide by \eqref{eq:EFchi} with the non trivial zeroes 
of $\Lambda(\chi,s)$) satisfy $s -1/2 + \overline{s} -1/2 =0$ or equivalently:
$\Re s = \frac{1}{2}.$ Therefore 
Deninger's formalism should imply the Riemann hypothesis for $\Lambda(\chi,s)$!
This argument comes from an idea of Serre \cite{Serre} and has been 
formalized in the foliation case in \cite{D-S}.
Of course, we have described only a very small part of Deninger's formalism which 
deals also with $L-$functions of motives, Artin conjecture, Beilinson conjectures....etc.

\medskip

\begin{comment} \label{com:disym} There is a dissymmetry in \eqref{eq:EF} and in \eqref{eq:EFchi} between 
the coefficients of $\alpha (k \log p)$ and $\alpha (-k \log p)$ for 
$k\in \N^*.$ In the framework of Deninger's formalism 
the explanation is the following. Equation \eqref{eq:Poincare}   
 implies  "formally" that 
the transpose of $e^{t \Theta_{1, \chi}}$ is $e^t e^{-t \Theta_{1, \overline{\chi}}}$. Therefore, if 
we have a Lefschetz cohomological interpretation of \eqref{eq:EF} in Deninger's formalism
for a test function $\alpha$ with support in $]0, +\infty[$ then 
we have also a cohomological proof of \eqref{eq:EF}  for $\alpha$ with support in $]-\infty, 0[.$
In this formalism, \eqref{eq:Poincare} (and the above dissymmetry) is 
quite connected to the Riemann hypothesis.
\end{comment}
\medskip

Recall that  Alain Connes \cite{C 2} has reduced the validity of the Riemann 
hypothesis (for the $L-$functions of the Hecke characters) to a trace formula.


\section{Analogy with the foliation case.}  \label{sect:analogy}$\;$

\subsection{The Guillemin-Sternberg trace formula.} $\;$

Consider a smooth compact manifold $X$ with a smooth action:
$$
\phi: X \times \R \rightarrow X,\; (x,t) \rightarrow \phi^t(x),
$$ so that $\phi^{t+s}= \phi^t \circ \phi^{s}$ for any $t, s \in \R.$ 
Let $D_y \phi^t$ denote (for fixed $t\in \R$)  the differential 
of the map $: y\in X \rightarrow \phi^t(y).$ One has:
$ \partial_s \phi^{t+s}_{| s=0} (y) = D_y \phi^t  ( \partial_s \phi^s_{| s=0} (y)).$ 

Consider also a smooth vector bundle $E\rightarrow X.$ Assume 
that $E$ is endowed with a smooth family of maps 
$$
\psi^t: (\phi^t)^*E \rightarrow E, \, t \in \R,
$$ satisfying the following cocycle condition:
$$
\forall u \in C^\infty(X ; E),\, \forall t , s \in \R,\; 
\psi^{s}( \psi^t(u \circ \phi^t) \circ \phi^{s}) =
\psi^{t+s} (u \circ \phi^{t+s} ).
$$ In other words,  we require that  the maps $K^t : u \rightarrow \psi^{t} (u \circ \phi^{t} )=K^t (u)$ define an action 
of the additive group $\R$ on $C^\infty(X ; E).$
Notice that in the case of $E=\wedge^* T^* X$ and $\psi^s = ^{t}D \phi^s$ (the transpose 
of the differential $D \phi^s$ of $\phi^s$),
this condition is satisfied.

We shall assume that the graph of $\phi$ (i.e $\{ (x , \phi^t(x), t)\}$) meets transversally 
the "diagonal" $\{(x,x,t),\; x \in X,\, t \in \R
\}.$ Guillemin-Sternberg have checked (\cite{G-S}) that the trace 
$Tr ( K^t | C^\infty(X ; E) )$ is defined as a distribution of $t \in \R\setminus \{0\} $  by the formula:
$$
Tr ( K^t | C^\infty(X ; E) ) = \int_X K^t(x,x)
$$ where $K^t(x,y)$ denote Schwartz (density) kernel of $K^t.$ We warn the reader that, 
in general, for  $\alpha \in C^\infty_{compact} ( \R) \setminus \{0\}$, the operator $\int_\R \alpha(t) K^t d t$ is not trace class.

 Now, we give the name $T^0_x= \partial_t \phi^t(x)_{t=0}\,  \R$ to the real line 
generated by the vector field $ \partial_t \phi^t(x)_{t=0}$ of $\phi^t$ at a point $x$ 
where $ \partial_t \phi^t(x)_{t=0}\not=0.$  

\begin{proposition}(Guillemin-Sternberg, \cite{G-S} ) \label{prop:G-S}The following formula 
holds in $\mathcal{D}^\prime (\R\setminus \{0\}).$
$$
Tr ( K^t \, | \,  C^\infty(X ; E) ) = \sum_\gamma  l(\gamma) \sum_{k \in \Z^*}  
\frac{ Tr ( \psi^{k   l(\gamma)}_{x_\gamma} \, ; \, E_{x_\gamma} ) }{ 
| det \,( 1 - D_{y}  \phi^{k l(\gamma)}(x_\gamma) \, ;\,
 T_{x_\gamma} X / T^0_{x_\gamma} )\, |\,  }\delta_{k l(\gamma)} +
$$
$$
\sum_x \frac{ Tr ( \psi^{t}_{x} \, ; \, E_{x} ) }{ 
| det \,( 1 - D_{y}  \phi^{t}(x ) \, ; \,
 T_{x} X  )  | }.
$$ In the first sum, $\gamma $ runs over the periodic primitive orbits of $\phi^t$, 
$x_\gamma$ denotes any point of $\gamma$, $l(\gamma)$ is the length of 
$\gamma,$ $\phi^{l(\gamma)}( x_\gamma) =x_\gamma.$  In the second sum, $x$ 
runs over the fixed points of the flow: $\phi^t(x)=x$ for any $t \in \R.$

\end{proposition}
\begin{comment} Recall that $D_y \phi^t$ denotes, for fixed $t$, the differential 
of the map $y(\in X) \rightarrow  \phi^t(y).$ The non vanishing of the two determinants 
in Proposition  \ref{prop:G-S} is equivalent to the fact that the graph of $\phi$ meets transversally 
the "diagonal" $\{(x,x,t),\; x \in X,\, t \in \R
\}.$
\end{comment}
Note that the following elementary observation is the main ingredient of the proof the
Proposition \ref{prop:G-S}.
Let $A \in GL_n (\R)$ and $\delta_0(\cdot)$ denote the Dirac mass at $0\in \R^n.$ 
Then one computes the distribution  $\delta_0(A \cdot)$ in the following way.
For any $f \in C^\infty_{comp}(\R^n),$ one has:
$$
< \delta_0 (A \cdot) ; f (\cdot) > =  \int_{\R^n} \delta_0(A x) f(x) d x=  
 \int_{\R^n} \delta_0( y) f( A^{-1} y)   \frac{1}{Jac (A)}d y = \frac{1}{Jac (A)}  f(0)
$$ where $d y$ denotes the Lebesgue measure.
Therefore: $\delta_0(A \cdot) =   \frac{1}{Jac (A)}  \delta_0 ( \cdot).$

\subsection{The Lesfchetz trace formula of Alvarez-Lopez and Kordyukov.} $\;$

\m

Now we shall assume that $X$ is a  compact three dimensional oriented  manifold and  endowed with
a codimension one  foliation $(X,\mathcal{F})$. We shall also assume that
 the flow 
 $\phi^t$ preserves the foliation    $(X,\mathcal{F})$, is transverse to 
 it and thus has no fixed point. Therefore  $(X,\mathcal{F})$  is a compact
 Riemannian foliation whose leaves are oriented. We shall apply later Proposition \ref{prop:G-S} 
 with $E=\wedge^* T^* \mathcal{F} \rightarrow X.$
 
 \begin{comment}
 A typical example is of the form
 $X= \frac{L \times \R^{+ *}}{ \Lambda}$, where $\Lambda$ a subgroup 
 of $(\R^{+*}, \times)$  and $\phi^t(l,x)= (l ,  x e^{-t})$. 
 \end{comment}
 
 Now, we get a so called bundle like metric $g_X$ on $(X,\mathcal{F})$ in the following 
 way. We require that $g_X( \partial_t \phi^t(z) )=1$,
  $ \partial_t \phi^t(z) \perp T \mathcal{F}$  for any $(t,z) \in \R \times X$,  and 
  that $(g_X)_{| T \mathcal{F}}$ is a given leafwise metric. By construction, 
with respect to $g_X$, 
  the foliation $(X,\mathcal{F})$ is defined locally by riemannian submersions.
  
  In this setting,  Alvarez-Lopez and Kordyukov \cite{A-K1} have proved the following 
  Hodge decomposition theorem ($0 \leq j \leq 2$):
  \begin{equation} \label{eq:Hodge}
  C^\infty(X, \wedge^j T^* \mathcal{F})= \ker \Delta^j_\tau \oplus \overline{ \Im \Delta^j_\tau}
  \end{equation}
   where $ \Delta^j_\tau$ denotes the leafwise Laplacian. Since we have 
   $ \frac{ \displaystyle \ker d_{\mathcal{F}}}{\overline{ \Im d_{\mathcal{F}}}}= \ker \Delta^j_\tau$, 
   we call the vector space  $H^j_\tau (X) =\ker \Delta^j_\tau $  a reduced 
   leafwise cohomology group.

   Let $\pi^j_\tau$ denote the projection of the vector space of leafwise differential forms $C^\infty(X, \wedge^j T^* \mathcal{F})$ onto  $H^j_\tau (X) =\ker \Delta^j_\tau $ according 
to \eqref{eq:Hodge}  with $0\leq j \leq 2$. Then  Alvarez-Lopez and 
Kordyukov  \cite{A-K2} have proved  the following Atiyah-Bott-Lefschetz trace formula.
\begin{theorem} \label{thm:AK} (\cite{A-K2}) Let $\alpha \in C^\infty_{compact}(\R).$ Then 
the operators 
$$
\int_\R \alpha(s) \,\pi^j_\tau\circ (\phi^s)^* \circ \pi^j_\tau \,d s
$$ are trace class for $0\leq j \leq 2$. Let $\chi_\Lambda$ denote the 
leafwise measured Connes Euler characteristic of $(X,\mathcal{F})$ (\cite{C 1}). Then one has:
\begin{equation} \label{eq:AK}
\sum_{j=0}^2 (-1)^j {\rm TR} \int_\R \alpha(s) \,\pi^j_\tau\circ (\phi^s)^* \circ \pi^j_\tau \,d s\,=
\chi_\Lambda \alpha(0) + \sum_{\gamma} \sum_{k \geq 1} l(\gamma)
\left( \, \epsilon_{-k \gamma} \alpha( -k l(\gamma) ) + \epsilon_{k \gamma} \alpha( k l(\gamma) )\, \right)
\end{equation} where 
$\gamma$ runs over the primitive closed orbits of $\phi^t$,  $l(\gamma)$ is the length of $\gamma$, 
$x_\gamma\in \gamma$ and $\epsilon_{\pm k \gamma}= {\rm sign} \, 
{\rm det} ({\rm id}- D\phi_{| {\rm T}_{x_\gamma} \mathcal{F}}^{\pm k l(\gamma)}).$ 
 
\end{theorem}
\begin{proof} (Sketch of the idea).The case where the support of $\alpha$ is included in a suitably 
small interval $[-\epsilon, +\epsilon]$ is treated separately. So let us assume that the (compact) support of $\alpha$ is included 
in $\R\setminus\{0\}$.
The authors show that the following quantity:
$$
H(t)= \sum_{j=0}^2 (-1)^j {\rm TR} \int_\R \alpha(s) \,e^{- t \Delta^j_\tau} \circ (\phi^s)^*  \,d s\
$$ does not depend on  the real $t>0$.
 Using non trivial arguments 
based on \eqref{eq:Hodge}, the authors then prove that 
$$
\lim_{t\rightarrow + \infty} H(t)= \sum_{j=0}^2 (-1)^j {\rm TR} \int_\R \alpha(s) \,\pi^j_\tau\circ (\phi^s)^* \circ \pi^j_\tau \,d s\,.
$$ On the other hand, they show that:
$$
\lim_{t\rightarrow 0^+} H(t)=\sum_{j=0}^2 (-1)^j {\rm Tr} \int_\R \alpha(s)  (\phi^s)^*  \,d s.
$$ But Proposition \ref{prop:G-S} (with $E= \wedge^j T^* \mathcal{F}$ and $\psi^s = ^{t}D \phi^s$) shows that the right handside is equal to:
$$
\sum_{\gamma} \,l(\gamma)\sum_{k \in \Z^*} \sum_{j=0}^2 (-1)^j\,
\frac{ {\rm Tr} \bigl( ^{t}(D_y \phi^{ k l(\gamma)}(x_\gamma)): \wedge^j {\rm T}^*_{x_\gamma}  \mathcal{F} \mapsto \wedge^j {\rm T}^*_{x_\gamma} \mathcal{F} \bigr)}
 {|{\rm det}( {\rm id } - D_y \phi^{ k l(\gamma)}_{| {\rm T}_{x_\gamma} \mathcal{F}}) |} 
  \alpha({ k l(\gamma) }).
$$  Then, using the equality $\lim_{t\rightarrow + \infty} H(t)=  \lim_{t\rightarrow 0^+} H(t)$, one then gets immediately the result.
\end{proof}
\begin{comment} If there exists a real $h_\gamma >0$ such that $ h_\gamma D\phi^{l(\gamma)}_{| {\rm T}_{x_\gamma} \mathcal{F}}$ 
belongs to $S0_2({\rm T}_{x_\gamma} \mathcal{F} )$ (ie $D\phi^{l(\gamma)}_{| {\rm T}_{x_\gamma} \mathcal{F}}$ 
is a direct similitude) then  $\epsilon_{\pm k \gamma}= 1$ for any integer $k$.
\end{comment}

\begin{comment}
The extension of Theorem \ref{thm:AK} to the  case where the flow has fixed points is the subject of a work in progress \cite{AKL}.
\end{comment}


\subsection{ The case of a ramified flat line bundle on $(X, \cF, (\phi^t)_{t\in \R} )$. } $\;$

\m

We consider now another compact (three dimensional) oriented riemannian foliation $(\tilde{X}, \cF, \phi^t)$ of codimension $1$ which defines a Galois ramified covering $\pi : \tilde{X} \rightarrow X$
with {\it finite } automorphism group $G$ such that the action of $G$ 
commutes with $\phi^t, t \in \R$ and permutes the leaves.  We assume that  for any real $t$, $\phi^t \circ \pi = \pi \circ \phi^t$ and that $\pi$ sends leaves onto leaves. 
We can also assume that $G$ preserves a  bundlelike metric $g'$ of $(\tilde{X}, \cF, \phi^t)$ and we fix such one.

Consider now a non trivial character $\rho: G \rightarrow S^1$. Define an action of $G$ on $\tilde{X} \times \C$ 
by setting
$$
\forall (h, m , \lambda) \in G \times \tilde{X} \times \C,\; h\cdot ( m , \lambda)= (h\cdot m, \rho^{-1}(h) \lambda)\,.
$$
To this action we associate the ramified flat complex line bundle 
$\cL_\rho = \frac{\tilde{X} \times \C}{G} \rightarrow X$ over $X$, where
any $(m, \lambda) \in \tilde{X} \times \C$  is  identified to $(h\cdot m, \rho^{-1}(h) \lambda)$ for 
any $h\in G$.

One defines a projection $P_j$ acting on 
$H^j_\tau ( \tilde{X}) \otimes_\R \C$ ($0 \leq j \leq 2 $) by setting:
$$
P_j= \frac{1}{| G |} \sum_{h \in G} h^*\, , \, {\rm card}\, G = | G |\, .
$$  
\begin{definition} One then defines the leafwise cohomology group $H^j_\tau ( X ; \cL_\rho)$ 
with coefficient in $\cL_\rho$ by:
$$
H^j_\tau ( X ; \cL_\rho) = \Im P_j \, , 0 \leq j \leq 2\,.
$$
\end{definition}

Since the flow $(\phi^t)$ commutes with $G$, it induces an action denoted $(\phi^t)_j^*$ ($= \pi^j_\tau  (\phi^t)^*)$ on each $H^j_\tau ( X ; \cL_\rho) .$

The duality between $\cL_\rho$ and $\cL_{\overline{\rho}}$ induces a map:
$$
H^1_\tau ( X ; \cL_\rho) \times H^{1}_\tau ( X ; \cL_{\overline{\rho}}) \rightarrow H^2_\tau ( X )
$$
$$
(\alpha, \beta) \rightarrow \alpha \wedge \beta\,.
$$

The leaves of $(\tilde{X}, \cF)$ (and of $(X, \cF)$) are oriented. The restriction of the  $G-$invariant metric $g'$ 
along the leaves  then induces the  Hodge star operator:
$$
H^j_\tau ( X ; \cL_\rho) \rightarrow H^{2-j}_\tau ( X ; \cL_{\overline{\rho}}) 
$$
$$
\omega \rightarrow \overline{ \star \omega} \,.
$$ 	


Consider a closed orbit $\gamma$ in $X$ defined by $t\in [0, T] \rightarrow \phi^t(x_0)$ 
where $\phi^T(x_0)= x_0$. We shall say that this closed orbit is ramified if 
the cardinal of $\pi^{-1}(x_0)$ is strictly smaller than the cardinal of $G$. Since the action of $G$ commutes with 
one of $\phi^t (t\in \R)$,  this definition does not depend on the choice of $x_0 \in \gamma$.
We shall assume that there are only a finite number of closed ramified orbits.
Moreover, for any such ramified closed orbit, we shall make the following two assumptions:

$\bullet$ First,  if $\tilde{x_0} \in \pi^{-1} ( \{ x_0\})$
then the restriction of $\rho$ to $G_{\tilde{x_0}}= \{r \in G /\, r\cdot \tilde{x_0}= \tilde{x_0}\}$ 
is not trivial. 

$\bullet$ Second, if $T>0$ and $h \in G$ are such that $ \phi^T(\tilde{x_0}) =  h\cdot \tilde{x_0}$, then 
\begin{equation} \label{eq:signcte}
\forall r \in G_{\tilde{x_0}},\; 
 {\rm sign\, det}\, \bigl( Id - D ( h^{-1} \, r \circ \phi^T)\bigr)_{| T_{\tilde{x_0}} \mathcal{F}} = {\rm sign\, det}\, \bigl( Id - D (h^{-1} \circ \phi^T)\bigr)_{| T_{\tilde{x_0}} \mathcal{F}}
\,.
 \end{equation}

\smallskip
Consider now the case where $\gamma$ is an unramified closed orbit on $X$. Let $\tilde{x_0} \in \pi^{-1} ( \{ x_0\})$, 
there exists $h \in G$ such that $\phi^T ( \tilde{x_0}) = h\cdot \tilde{x_0}$. Then, for any $\lambda \in \C$,
$$
h^{-1} \cdot ( \phi^T ( \tilde{x_0}) , \lambda)= ( \tilde{x_0} ,  \rho(h) \lambda)\,.
$$ 
Then  the complex number $\rho(h)$ defines
the monodromy action along $\gamma$ on the flat line bundle $\cL_\rho$ and we denote it by $\rho(\gamma)$.

We then may state a leafwise Lefschetz trace formula (with coefficients in $\cL_\rho$) 
where, in analogy with the ramified primes of a Dirichlet character, the ramified closed orbits do not contribute at all.

\begin{theorem} \label{thm:traceram} Assume that for any $h \in G$,  the graph of $h \circ \phi^t$  
intersects transversally the "diagonal"  $\{( \tilde{x},  \tilde{x}, t) /  \; \tilde{x}\in \tilde{X} , \, t \in \R \}$. Let $\alpha \in C^\infty_{compact}(\R)$ be such that 
$\alpha(0)=0$. Then 
for each $0\leq j \leq 2$, the operator $\int_\R \alpha(s)  (\phi^s)_j^* \,d s$ acting on $H^j_\tau ( X ; \cL_\rho)$ 
is  trace class. 
 Moreover,  one has:
$$ 
\sum_{j=0}^2 (-1)^j {\rm TR} \int_\R \alpha(s) \, (\phi^s)_j^*  \,d s\,=
$$
$$
\sum_{\gamma} \sum_{k \geq 1} l(\gamma)
\bigl( \, \epsilon_{-k \gamma} \,\rho( -k \gamma) \,\alpha( -k l(\gamma) ) + \epsilon_{k \gamma}\, \rho( k \gamma)\, \alpha( k l(\gamma) )\, \bigr)
$$ where 
$\gamma$ runs over the primitive unramified closed orbits of $\phi^t$,  $l(\gamma)$ is the length of $\gamma$, 
$x_\gamma\in \gamma$ and $\epsilon_{\pm k \gamma}= {\rm sign} \, 
{\rm det} ({\rm id}- D\phi_{| {\rm T}_{x_\gamma} \mathcal{F}}^{\pm k l(\gamma)}).$ 

\end{theorem}
\begin{proof} We  use Theorem \ref{thm:AK} or rather its proof to compute, the alternate 
sum of traces:
\begin{equation} \label{eq:alttr}
\sum_{j=0}^2 (-1)^j \frac{1}{| G|} \sum_{h\in G}\, {\rm TR}\, \int_\R \alpha(s) \, \pi^j_\tau \,(h^{-1}\circ \phi^s)^*  \,d s : H^j_\tau (\tilde{X})\otimes_\R \C \rightarrow H^j_\tau (\tilde{X}) \otimes_\R \C\,.
\end{equation} So we have to consider the reals $T\not=0$ and the points $\tilde{x}_0 \in \tilde{X}$ such that 
$h^{-1}\circ \phi^{T} ( \tilde{x}_0) = \tilde{x}_0 $. We shall assume $T>0$, the case 
$T<0$ being similar. By considering $\pi ( h^{-1}\circ \phi^{t} ( \tilde{x}_0) )$, one obtains
a closed orbit on $X$, $\gamma_{ \pi(\tilde{x}_0)} : t \mapsto \phi^t( \pi ( \tilde{x}_0))$ ($ 0 \leq t \leq T$) of length $T$. There exists 
$k \in \N^*$, such that $\gamma_{ \pi(\tilde{x}_0)}$ is the $k-$iteration of the {\it primitive} closed orbit 
of length $T_0 = \frac{T}{k}$ determined by $\phi^{T_0} ( \pi ( \tilde{x}_0)) = \pi ( \tilde{x}_0)$. Upstairs on $\tilde{X}$, this means that 
there exists $h_0\in G$ such that $h_0^{-1} \phi^{T_0} (  \tilde{x}_0 ) = \tilde{x}_0 $.  So $\phi^{ k T_0} (  \tilde{x}_0 ) =
h_0^k \cdot \tilde{x}_0 = h \cdot \tilde{x}_0$ and hence
$h^{-1}h_0^k \in G_{\tilde{x_0}}$. We then distinguish two cases.

A) The case where $\gamma_{ \pi(\tilde{x}_0)}$ is unramified.

 Then by considering the translates on the left of  
 $h^{-1} \circ \phi^t (  \tilde{x}_0 )$, one obtains  exactly the  $| G |$ curves (of the flow) on $\tilde{X}$ which correspond 
(via $\pi$) to $\gamma_{\pi( \tilde{x}_0)}$. We write them in the following way, $t \mapsto l h^{-1} l^{-1}\circ \phi^t ( l \cdot \tilde{x}_0 )$ 
( $ 0 \leq t \leq T$), where $l \in G$, the corresponding monodromy being $\rho(l h l^{-1})$. Observe that 
$l \cdot \tilde{x}_0$ is a fixed point of $ l h^{-1} l^{-1}\circ \phi^T$.
Then the proof of Theorem  \ref{thm:AK} (\cite{A-K2}) shows that the geometric contribution of $\gamma_{\pi( \tilde{x}_0)}$ to \eqref{eq:alttr} is computed according to Proposition \ref{prop:G-S} and is equal to:
$$
\frac{T_0}{| G|} \sum_{l \in G} \sum_{j=0}^2 (-1)^j \frac{ Tr\, \bigl( ^{t}D (l h^{-1} l^{-1}\circ \phi^T ) ( l\cdot \tilde{x_0})\, ;\, \wedge^j T^*_{l\cdot \tilde{x_0}} \mathcal{F}  \bigr)}
{| \,{\rm  det}\, \bigl( {\rm id} - D (l h^{-1} l^{-1}\circ \phi^T)\bigr)_{| T_{l\cdot \tilde{x_0}} \mathcal{F}}\, |}\,\,\rho(l h l^{-1}) \, \alpha( k T_0)\, .
$$
We observe that all the following  reals, where $l$ runs over $G$, have the same sign:
$$
{\rm  det}\, \bigl( {\rm id} - D (l h^{-1} l^{-1}\circ \phi^T)\bigr)_{| T_{l\cdot \tilde{x_0}} \mathcal{F}}, \; {\rm det} ({\rm id}- D\phi_{| {\rm T}_{\pi(\tilde{x}_0)} \mathcal{F}}^T)\,.
$$
Therefore, the previous expression is clearly equal to $ T_0 \,\epsilon_{k \gamma_{ \pi(\tilde{x}_0)}} \, \rho(h) \,\alpha( k T_0)$ which 
yields the desired contribution.

B) The case where $\gamma_{ \pi(\tilde{x}_0)}$ is ramified.

So $G_{\tilde{x_0}}= \{u \in G /\, u\cdot \tilde{x_0}= \tilde{x_0}\}$ 
is not trivial. Then there are exactly $| G/ G_{\tilde{x_0}}|$ curves (of the flow) upstairs on $\tilde{X}$ which correspond 
(via $\pi$) to $\gamma_{\pi( \tilde{x}_0)}$. 
They are given by $t \mapsto l_j h^{-1} l_j^{-1}\circ \phi^t ( l_j \cdot \tilde{x}_0 )$ 
( $ 0 \leq t \leq T$), where the $l_j$ ($1\leq j \leq m$) run over a system of representatives of cosets of 
$ G/ G_{\tilde{x_0}}$. Since the restriction of $\rho$ to $G_{\tilde{x_0}}$ is {\it not} trivial, we have to count
each such curve $| G_{\tilde{x_0}}|$ times but with (possibly) 
different monodromies (ie action on the line factor $\C$). More precisely,  for each representative $l_j$  the curve labeled
\begin{equation} \label{eq:rami}
t \mapsto   l_j h^{-1} l_j^{-1} u\circ \phi^t ( l_j \cdot \tilde{x}_0 ),\; {\rm with}\, u \in l_j G_{\tilde{x_0}} l_j^{-1}
\end{equation} has monodromy $\rho(u^{-1} l_j h l_j^{-1} )$.

Thanks to the sign assumption \eqref{eq:signcte}, the proof given above in the unramified case shows that the sum of the contributions
to \eqref{eq:alttr} of the $|l_j G_{\tilde{x_0}} l_j^{-1} |$ curves in \eqref{eq:rami} is then equal to: 
$$
C \sum_{u \in l_j G_{\tilde{x_0}} l_j^{-1}} \rho(u^{-1} l_j h l_j^{-1} ) = C  \bigl( \,\sum_{s \in  G_{\tilde{x_0}} } \rho(  s) \, \bigr)\, \rho(  h) \,,
$$ where $C$ is a suitable constant.
But since the restriction of $\rho$ to $G_{\tilde{x_0}}$ is assumed to be not trivial, this contribution is zero.

\end{proof}

\m
\subsection{An explicit example.}$\;$

\m

Now we describe an explicit example (communicated to us by Jesus Alvarez-Lopez) of a ramified covering
satisfying the conditions of Theorem \ref{thm:traceram}. Denote by  $\mathcal{S}$
the Jacob ladder, a certain noncompact surface embedded in $\R^3$, and by $L$ the real line 
of symmetry of $\mathcal{S}$. There is a group $\mathcal{T}$ of translations isomorphic to 
$(\Z , +)$, whose vectors belong to $L$ and which acts on $\mathcal{S}$ such that  
the quotient $\mathcal{S}/\mathcal{T}$ is a smooth compact Riemann surface. Let $G= \{ Id, R\} $ denote the group 
generated by the rotation $R$ of $\R^3$ whose axis is $L$ and angle is $\pi$. 
Observe that we have a ramified covering $\mathcal{S}\rightarrow \mathcal{S}/G$ where the set of ramification points 
is $L \cap \mathcal{S}$. Moreover, there exists a vector field $U$ on $\mathcal{S}$ whose fixed points are exactly 
the ones of the $G-$action and which is invariant by $G$ and $\mathcal{T}$. Consider an action 
of $\mathcal{T} \simeq (\Z, +)$ on the circle $S^1$ defined by a rotation of angle $2 \pi \alpha$ ($\alpha \notin \Q$).
Now, set $\tilde{X}=\displaystyle \frac{\mathcal{S}\times S^1}{\mathcal{T}}$, it is foliated by the leaves induced by the sets $ \mathcal{S}\times \{ e^{i \theta}\}$.
Consider $\phi^t$ 
the flow of the vector field $U \times \frac{\partial}{\partial_x}$ of $\tilde{X}$. Then we can choose $U$ such that 
the hypothesis of Theorem \ref{thm:traceram} are satisfied by $\tilde{X}, (\phi^t)_{t \in \R},  X= \tilde{X}/G$.

\b


%

\b

\subsection{The more general case of a flat ramified complex vector bundle.}$\;$


\m

More generally, one can consider a unitary representation $\rho : G \rightarrow U( E) $
where $E$ is a complex hermitian vector space.  
One then gets the ramified flat complex hermitian vector bundle 
$$\cE_\rho = \frac{\tilde{X} \times E}{G} \rightarrow X
$$ over $X$ where $(m,v)$ is identified with $(h\cdot m, \rho(h)^{-1} v)$ for any $h \in G$. Similarly, 
the dual representation $^{t}\rho^{-1}: G \rightarrow U(E^*)$ allows 
to consider the dual flat complex hermitian vector space $\cE^*_{^{t}\rho^{-1}} \rightarrow X$.

The duality between $\cE_\rho$ and $\cE^*_{^{t}\rho^{-1}}$ induces a map:
$$
H^1_\tau ( X ; \cE_\rho) \times H^{1}_\tau ( X ; \cE^*_{^{t}\rho^{-1}}) \rightarrow H^2_\tau ( X )
$$
$$
(\alpha, \beta) \rightarrow \alpha \wedge \beta\,.
$$

 Denote by  $J:\, \cE_\rho \rightarrow \cE^*_{^{t}\rho^{-1}}$ 
the antilinear vector bundle isomorphism
provided by the hermitian scalar product.  Then, using leafwise Hodge star of the  metric $g'$, 
one gets the following  Hodge star operator acting on the cohomology groups:
$$
H^j_\tau ( X ; \cE_\rho) \rightarrow H^{2-j}_\tau ( X ; \cE^*_{^{t}\rho^{-1}}) 
$$
$$
\omega \rightarrow {J \star \omega} \,.
$$ 	



%

 \section{Remarks about a conjectural dynamical laminated foliated space $ (S_K, \mathcal{F}, g, \phi^t)$
associated to the Dedekind zeta function $\widehat{\zeta}_K $.}    \label{sect:Riemann} $\;$
 
 \m
 
 Much of the following Section is speculative in nature. It should be viewed as a working programme or a motivation 
 for developing interesting mathematics.
 
 Let $K$ be a number field and let $\mathcal{O}_K$ denote its ring of integers. Let $r_1$ (resp. $2 r_2$)
 denote the number of real  (resp. complex) embeddings of $K$ so that 
 the dimension of $K$ as a $\Q-$vector space is equal to $r_1 + 2 r_2$. If $\sigma: K \rightarrow \C$ is  a complex 
 embedding then of course $ | \sigma (z) | $ and $ | \overline{\sigma} (z) | $ define the same archimedean absolute 
 value on $K$. Therefore the set $S_\infty$ of all the  archimedean absolute values of $K$ has exactly $r_1+r_2$ elements.
 
 We now set:
 $$
 \Gamma_\R(s)= \pi^{-s/2} \Gamma (s/2) ,\; \Gamma_\C(s)= (2\pi)^{-s} \Gamma(s)\,.
 $$
 
 The Dedekind zeta function  $\widehat{\zeta}_K$ is defined for $\Re s >1$ by:
 $$
 \widehat{\zeta}_K (s) = | d_K |^{s/2} \,\Gamma_\R^{r_1}(s)\, \Gamma_\C^{r_2}(s)\,
 \prod \frac{1}{1- (N \mathcal{P})^{-s}}\, ,
 $$ where $d_K$  denotes the discriminant of $K$ over $\Q$, $ \mathcal{P}$ runs over the set of non zero prime ideals 
of $\mathcal{O}_K$ and $ N \mathcal{P}$ ($=$ card $  \mathcal{O}_K/ \mathcal{P} $) denotes the norm of $\mathcal{P}.$ 

The function $\widehat{\zeta}_K$ 
extends as a meromorphic function on $\C$ and admits a simple pole at $0$ and $1$. It  satisfies the functional 
equation $\widehat{\zeta}_K (s) = \widehat{\zeta}_K(1-s)$.

We recall the  explicit formula for the zeta function  $\widehat{\zeta}_K$  as an equality
between two distributions in $\mathcal{D}^\prime (\R\setminus \{0\})$ ($t$ being 
the real variable).
\begin{equation} \label{eq:EFK}
\begin{aligned}[rcl]
1 - & \sum_{\rho \in \widehat{\zeta}_K^{-1}\{0\}, \; \Re \rho \geq 0} 
e^{ t \rho} + e^t= 
\sum_{{\mathcal{P}} }\log N \mathcal{P}  \sum_{k\geq 1} (
\delta_{k \log N \mathcal{P}} + (N \mathcal{P})^{-k} \delta_{-k \log N \mathcal{P}} ) \\&
+ r_1 \bigl( \frac{1}{1-e^{-2t}} \,1_{\{t>0\}} + \frac{e^t}{1-e^{2t}} \,1_{\{t<0\}} \bigr)+ r_2 \bigl( \frac{1}{1-e^{-t}} \,1_{\{t>0\}} + \frac{e^t}{1-e^{t}} \,1_{\{t<0\}} \bigr),
\end{aligned}
\end{equation}
 where $ \mathcal{P}$ runs over the set of prime ideals 
of $\mathcal{O}_K$ and $ N \mathcal{P}$  denotes the norm of $\mathcal{P}.$

 \subsection{ Structural assumptions and their consequences.}$\;$
 
 \m
 
 We assume,  following Deninger (eg \cite{D 1c}, \cite{D 2b}),  that to $Spec \,\mathcal{O}_K \cup S_\infty$, one can associate a
  (laminated) foliated space 
$ (S_K, \mathcal{F}, g, \phi^t)$ satisfying the following assumptions.

{\bf 1].} The leaves are Riemann surfaces, the path connected components
of $S_K$ are three dimensional and ,
 $g$ denotes a leafwise riemannian 
metric. 
The flow  $(\phi^t) _{t\in\R}$  acts on 
$ (S_K, \mathcal{F})$, it sends  leaves into  (other) leaves and its  graph  
intersects transversally the "diagonal"  $\{( \tilde{x},  \tilde{x}, t) /  \; \tilde{x}\in S_K , \, t \in \R \}$.

{\bf 2].} 
To each prime ideal $\mathcal {P}$ of $ \mathcal{O}_K$ there corresponds a unique primitive closed 
orbit $\gamma_\mathcal {P}$ of $\phi^t$ of length $\log N \mathcal {P}.$ 
There is a bijection between  the set $S_\infty$ of archimedean absolute values   
and the set of fixed point $y_\infty=\phi^t(y_\infty), \forall  t\in \R,$ of the 
flow. Each leaf contains at most one fixed point and the flow is transverse to all the leaves different from the 
ones containing the $r_1+ r_2$ fixed points.

{\bf 3] a).} 

 We  assume  that for any fixed point $y_\infty$:
\begin{equation} \label{eq:SO} 
\forall t \in \R,\; 
e^{-t/2} D_y \phi^t(y_\infty)_{| T_{y_\infty}  \mathcal{F}} \in  
SO_2 ( T_{y_\infty}  \mathcal{F} )\,.
\end{equation}
 
 {\bf  3] b).} For any prime $\mathcal {P}$ of $ \mathcal{O}_K$  and any $\tilde{x} \in \gamma_\mathcal {P}$:
 
\begin{equation} \label{eq:SOb} 
 \displaystyle e^{- \frac{\log N\mathcal {P}}{2} } D_y \phi^{\log N\mathcal {P} }(\tilde{x} )_{| T_{\tilde{x}}  \mathcal{F}} \in  
SO_2 ( T_{\tilde{x}}  \mathcal{F} )\,.
\end{equation}
\m

{\bf 4].}  We have (Frechet) reduced  real leafwise cohomology groups $\overline{H}^j_{\mathcal {F} , K}$ 
($0\leq j \leq 2$), on which $(\phi^t) _{t\in\R}$ acts naturally, 
 with the following properties. One has 
 $\overline{H}^0_{\mathcal {F} , K}\simeq \R$ (the space of constant functions) and $\overline{H}^2_{\mathcal {F} , K} \simeq \R [\lambda_g]$ where $[\lambda_g]$ denote the class in  $\overline{H}^2_{\mathcal {F} , K}$ 
of the leafwise kaehler  metric  $ \lambda_g$  associated to $g.$
Moreover, we assume that 
\begin{equation} \label{eq:Henri}
\forall t \in \R,\; (\phi^t)^* ([\lambda_g])= e^t [\lambda_g]\,,
\end{equation} and that $\overline{H}^1_{\mathcal {F} , K}$ 
is infinite dimensional.

{\bf 5]} The action of $\phi^t$ on $\overline{H}^1_{\mathcal {F} , K}$ commutes 
with the Hodge star $\star$ induced by $g.$ Moreover there exists a transverse 
measure $\mu$ on $(S_K, \mathcal{F})$ such that 
$\int_{S_K} ( \alpha \wedge \star \beta ) \mu $ defines 
a scalar product on $\overline{H}^1_{\mathcal {F} , K}$

{\bf 6].}
For any $\alpha \in C^\infty_{compact}(\R \setminus \{0\} ; \R)$,
$\int_\R \alpha(t) (\phi^t)^* d t$ acting on $\overline{H}^1_{\mathcal {F} , K}$  
is trace class (possibly in some generalized sense, cf \cite{AKL}). 
 The
explicit formula \eqref{eq:EFK} is  interpreted as an Atiyah-Bott-Lefschetz trace formula for 
the   foliated space 
$ (S_K, \mathcal{F}, g, \phi^t)$  with respect to the leafwise 
cohomology groups $\overline{H}^j_{\mathcal {F} , K}$ 
($0\leq j \leq 2$). In particular, the infinitesimal generator $\theta_1$ 
of $(\phi^t)_{t\in \R}$ acting on $\overline{H}^1_{\mathcal {F} , K}\otimes \C$ has discrete spectrum, its set of
eigenvalues coincide with the set of zeroes of  $\widehat{\zeta}_K (s) $ (with the same multiplicities on each side).
Moreover:
$$
\overline{H}^1_{\mathcal {F} , K}\otimes_\R \C = \overline{ \sum_{z_q \in \widehat{\zeta}_K^{-1}\{0\} }
\ker (\theta_1 - z_q Id)^{n(z_q)} } \,.
$$

{\bf 7].} Let  $x_\infty \in S_K$  be any fixed point corresponding (according to 2]) to 
a real archimedan absolute value. Then $x_\infty$ should be a limit point 
of a trajectory $\gamma_\infty:$ $\lim_{t \rightarrow +\infty} \phi^t(y) =x_\infty$
for any $y \in \gamma_\infty.$ Moreover, $\gamma_\infty$ should 
have the following orbifold structure. Define an orbifold structure 
on $\R^{\geq 0}$ by requiring the following map to be 
an orbifold isomorphism:
$$
Sq: \frac{\R}{\{ 1,  -1\}} \rightarrow \R^{\geq 0},\; Sq(z)=z^2.
$$ Notice that $Sq$ transforms the flow 
$\phi^t_{ \frac{\R}{\{ 1,  -1\}} }(z)= z e^{-t}$ into 
the flow $\phi^t_{\R^{\geq 0}}(v)= v e^{-2 t}.$ 
Then we require that there exists an embedding 
$\Psi:  \R^{\geq 0} \rightarrow \gamma_{\infty} $ 
such that $\Psi (0)=x_\infty$ and  
\begin{equation} \label{eq:orb}
\forall (t, v)\in \R \times\R^{\geq 0},\; \Psi (\phi^t_{\R^{\geq 0}}(v)=v e^{-2 t } )= \phi^t( \Psi(v)).
\end{equation} 
Lastly we require that $\gamma_\infty$ is transverse at $x_\infty$ to 
$T_{x_\infty} \mathcal{F}.$ 

\m

{\bf 8].}  Let $z_\infty 
\in S_K$ be any fixed point corresponding (according to 2]) to a complex archimedean absolute value. Then there exist  two trajectories $\gamma_\pm$ 
of the flow $\phi^t$ with end point $z_\infty.$
For any $z_\pm \in \gamma_\pm$, $\lim_{t \rightarrow +\infty} \phi^t(z_\pm)= z_\infty.$ These two 
trajectories $\gamma_\pm$ are transverse to $\mathcal{F}$ at 
$z_\infty $. Moreover there exists an embedding:
$$
\Psi: \R \rightarrow \gamma_- \cup \gamma_+,
$$ such that $\Psi(0)= z_\infty$, $\gamma_\pm \setminus \{0\}=
\Psi( \R^\pm \setminus \{0\}).$ Lastly, $\forall v,t \in \R, $
$\Psi ( v e^{-t})= \phi^t( \Psi(v)).$

\bigskip

\begin{comment} The stronger assumption $ \forall t \in \R,\, (\phi^t)^* (g)= e^t g$ implies 
\eqref{eq:SO} (because $\phi^0= Id$), \eqref{eq:Henri} and the fact that 
$\phi^t$ commutes with the Hodge star not only  on $\overline{H}^1_{\mathcal {F} , K}$ but also
on the vector space of leafwise differential $1-$forms.  Deninger 
told us privately that this assumption  $(\phi^t)^* (g)= e^t g$ might be too strong.
Assumption 5] and  \eqref{eq:Henri} implies the analogue of Equation \eqref{eq:hilb} for $\widehat{\zeta}_K$ in 
Deninger's formalism.  Therefore, the first six Assumptions imply 
the Riemann hypothesis  for $\widehat{\zeta}_K$ as explained in Section 2 !.  
\end{comment}
\begin{comment}   The disymmetry mentionned in Comment \ref{com:disym} might be explained in the following way.
 For each prime ideal $\mathcal{P}$ of $\mathcal{O}_K$ with norm $p^f$, 
   $(S_K, \mathcal{F})$ should 
  exhibit a transversal of the type $]0,1[ \times \Z_p$ and possibly 
  the  ring of finite Adeles $\A_\Q$ might enter into the picture. See \cite{El1} for a simple case. 
\end{comment}

\subsection{Remarks about the contribution of the archimedean places in \eqref{eq:EFK}.}$\;$ 

\bigskip

Now we apply formally  the Guillemin-Sternberg trace formula 
for the distribution of the real variable $t$:
\begin{equation} \label{eq:GSK}
\sum_{j=0}^2 (-1)^j Tr(  (\phi^t)^*\, ; \, \Gamma (S_K\, ; \,  \wedge^j T^* \mathcal{F}) )
\end{equation} where $ \Gamma (S_K\, ; \,  \wedge^j T^* \mathcal{F})$ denotes 
the set of  "smooth" sections. 
\begin{lemma} \label{lem:cont} (Deninger \cite{D 2b}) {\item 1]} The contribution of a fixed point $y_\infty$  corresponding 
to an archimedean place of $K$
in the Guillemin-Sternberg trace formula for \eqref{eq:GSK} is:
$$
 \frac{1}{| det ( 1 -D_y \phi^t(y_\infty) \, ; \, T_{y_\infty}  S_K/T_{y_\infty} \mathcal{F}) | }.
$$
{\item 2]}  In the case of a fixed  point $x_\infty$  corresponding 
to a real  archimedean place of $K$ one has:
$$
 \forall t \in \R\setminus \{0\},\;  \frac{1}{| det ( 1 -D_y \phi^t(x_\infty) \, ; \, T_{x_\infty}  S_K/T_{x_\infty} \mathcal{F}) | } = \frac{1}{| 1-e^{-2 t} |}
 $$
 {\item 3]}  In the case of a fixed  point $z_\infty$  corresponding 
to a complex  archimedean place of $K$ one has:
$$
 \forall t \in \R\setminus \{0\},\;  \frac{1}{| det ( 1 -D_y \phi^t(z_\infty) \, ; \, T_{z_\infty}  S_K/T_{z_\infty} \mathcal{F}) | } = \frac{1}{| 1-e^{- t} |}
 $$
\end{lemma}
\begin{proof} 1]
Using Proposition \ref{prop:G-S}, one sees that the contribution of the  fixed point $y_\infty$  
is equal to:
$$
\frac{ \sum_{j=0}^2 (-1)^j Tr(  (D_y \phi^t)^*(y_\infty)\, ;\, \wedge^j T_{y_\infty} ^* \mathcal{F} )}{
| det ( 1 -D_y \phi^t(y_\infty) \, ; \, T_{y_\infty} S_K ) | }=
$$
$$
 \frac{ det ( 1 -D_y \phi^t(y_\infty) \, ; \, T_{y_\infty}  \mathcal{F})  }{| det ( 1 -D_y \phi^t(y_\infty) \, ; \, T_{y_\infty}  \mathcal{F} ) |} \, \frac{1}{| det ( 1 -D_y \phi^t(y_\infty) \, ; \, T_{y_\infty}  S_K/T_{y_\infty} \mathcal{F}) | }.
$$
Using property 
 \eqref{eq:SO} one checks easily that 
 $$
  \frac{ det ( 1 -D_y \phi^t(y_\infty) \, ; \, T_{y_\infty}  \mathcal{F})  }{| det ( 1 -D_y \phi^t(y_\infty) \, ; \, T_{y_\infty}  \mathcal{F} ) |} =1.
 $$ One then gets immediately 1].

 \noindent 2] Since $T_{x_\infty}  S_K/T_{x_\infty} \mathcal{F} $ is a real line, there exists $\kappa \in \R$ 
such that:
$$
\forall t \in \R,\;  {| det ( 1 -D_y \phi^t(x_\infty) \, ; \, T_{x_\infty}  S_K/T_{x_\infty} \mathcal{F}) | }= 
{|1 - e^{\kappa t} |}.
$$
By Assumption 7], $\gamma_\infty$ is transverse at $x_\infty$ to 
$T_{x_\infty} \mathcal{F}$ and  \eqref{eq:orb} shows that 
$D_y \phi^t(x_\infty)$ acts 
as 
$e^{- 2 t} \,Id$ on the real line $   T_{x_\infty}  S_K/T_{x_\infty} \mathcal{F}$. One then gets 
2] immediately. One proves 3] in the same way, using Assumption 8].

\end{proof} 

 Recall that  we wish to test the interpretation of \eqref{eq:EFK} as a Lefschetz trace formula via 
the Guillemin-Sternberg formula. Part 2] of the next Proposition is a priori embarrassing...

\begin{proposition} (Deninger \cite{D 2b})\label{prop:not}
{\item 1]} The contribution 
of a real archimedean absolute value in \eqref{eq:EFK} 
coincides  {\it for any real   positive $t$} with the contribution of the corresponding fixed 
point  $x_\infty$ in the Guillemin-Sternberg formula for \eqref{eq:GSK}.

{\item 2]} The contributions of the fixed point 
$x_\infty$ {\it for $t$ real negative} in the Guillemin-Sternberg formula for \eqref{eq:GSK}
and of the corresponding real archimedean absolute value in  \eqref{eq:EFK} do {\it  not} coincide. 
(This riddle will be resolved in Section 4.4).
\end{proposition}

\begin{proof} 1] This is part 2] of the previous Lemma.

\noindent 2] Indeed, 
the Guillemin-Sternberg formula gives 
$$
\frac{1}{| 1- e^{-2 t} |}= \frac{e^{2 t}}{ 1- e^{2 t} },
$$ whereas \eqref{eq:EFK} gives $ \displaystyle \frac{e^{t}}{ 1- e^{2 t} }$ for $t<0.$


\end{proof} 
\begin{comment} It was in order to explain the factor $-2$  (instead of $-1$) in 
$\frac{1}{1-e^{-2 t}}$ for $t>0$ for a real archimedean place in \eqref{eq:EFK} that  Deninger has proposed in [Section 3]\cite{D 1c}  the  Assumption 7].
\end{comment}
The following Proposition 
shows that the contribution of a {\it complex } archimedean place in 
the explicit formula \eqref{eq:EFK} is better understood than the one of a {\it real} archimedean place 
(cf Prop. \ref{prop:not}. 2]).

\begin{proposition} ([Section 5]\cite{D 2b}) \label{prop:=}
 Let $z_\infty$ be a fixed point corresponding to a complex archimedean 
place of $K$.  The contribution of $z_\infty$ in the Guillemin-Sternberg 
trace formula (for \eqref{eq:GSK}) coincides, for $t\in \R\setminus \{0\},$ with 
the contribution of the corresponding complex archimedean place in \eqref{eq:EFK}.

\end{proposition}
\begin{proof} This is an easy consequence of Lemma \ref{lem:cont}. 3]. 

\end{proof}
\medskip
\medskip

\subsection{ More precise assumptions when  $K $ is a Galois extension of $ \Q$ with Galois group $G$.}$\;$
\medskip

In this subsection we assume that $K$ is a Galois extension of $\Q$ 
of degree $n$ with Galois group $G$. 
We then require the existence of a ramified Galois covering map
$\pi_K: S_K \rightarrow S_\Q$ whose automorphism group  coincides with $G$ and which  {\it satisfies the following properties}. 

For any real $t$, $\phi^t \circ \pi_K = \pi_K \circ \phi^t$, the map $\pi_K$ [resp. $G$] sends leaves to leaves. The leafwise 
metric $g$ is assumed to be $G-$invariant
 and, 
the actions of $G$ and $\phi^t \, (t\in \R) $ on $S_K$ and $\overline{H}^1_{\mathcal {F} , K}$ commute.
 Moreover the action of $G$
on the set of  curves $\gamma_{\mathcal{P}}$ coincides with the action of $G$ on the set of (non zero) prime ideals 
$\mathcal{P}$. More precisely, consider a prime number $p$ and the decomposition of 
$p \mathcal{O}_K$ in prime ideals:
$$
p \mathcal{O}_K = \mathcal{P}_1^e \ldots \mathcal{P}_r^e\,.
$$ Therefore, $n= e \,r\,f$ where $ N \mathcal{P}_j = p^f$ for $j\in \{1, \ldots, r\}$.

Consider $D_{\mathcal{P}_j}= \{h \in G /\, h \cdot \mathcal{P}_j= \mathcal{P}_j \}$ \footnote{ When $G$ is abelian, the $D_{\mathcal{P}_j}$ are all equal for $1\leq j \leq r$.} ,  
we then have a natural surjective homomorphism:
\begin{equation} \label{eq:Fr}
\begin{aligned}[rcl]
&
\Theta_j: D_{\mathcal{P}_j} \rightarrow {\rm Aut} \, \frac{ \mathcal{O}_K}{\mathcal{P}_j}\\&
h \mapsto \Theta_j(h)= ({\rm Fr\,})^{a_j(h)}\,
\end{aligned}
\end{equation} where ${\rm Fr\,}$ denote the Frobenius automorphim
$z \mapsto z^p$ of $\frac{ \mathcal{O}_K}{\mathcal{P}_j}$ and $a_j(h)$ is a suitable integer (modulo $f$). 
One has $ | D_{\mathcal{P}_j} | = e f$. The fact that 
$\Theta_j$ is natural means that
\begin{equation}\label{eq:Nat}
\forall (h  , v) \in G \times \mathcal{O}_K ,\; \; h\cdot v  - v^{p a_j(h)}  \in \mathcal{P}_j\,.
\end{equation}
Recall that $e$ is the common cardinal of the inertia groups $I_{\mathcal{P}_j} = \ker \Theta_j$ and that $e\geq 2$ if and only if 
$p$ divides the discriminant of $K$. 
We then require that   each point of $\gamma_{\mathcal{P}_j}$ 
is fixed by  $I_{\mathcal{P}_j}$ and that $\pi_K$ induces 
the  covering map
\begin{equation} \begin{aligned}[rcl]
& \pi_K: \, \gamma_{\mathcal{P}_j} \simeq \frac{\R}{f \log p\, \Z} \rightarrow 
\gamma_p \simeq \frac{\R}{ \log p\, \Z} \\ &
x \mapsto x\,.
\end{aligned}
\end{equation} 

Moreover, we require the following three properties:

\begin{equation}\label{eq:galoiscov}
\forall h \in D_{\mathcal{P}_j},\; \,\forall x \in \gamma_{\mathcal{P}_j} \simeq \frac{\R}{f \log p\, \Z},\; \;h\cdot x = x + a_j(h) \log p\, .
\end{equation} 

\noindent The restriction of $\phi^t$ to $\gamma_{\mathcal{P}_j} \simeq \frac{\R}{f \log p\, \Z}$ (resp. $\gamma_p $) is assumed to be the translation by $t$: $\phi^t(x)= x+ t$.
Lastly we   state a strengthening of Assumption 3] a):
\begin{equation} \label{eq:strength}
{\rm if}\, h^{-1} \phi^T( \tilde{x})=\tilde{x} \,\, {\rm for}\,\, \tilde{x}\in S_K,\, h\in G,\; 
{\rm then} \, D( h^{-1} \phi^T)(\tilde{x})_{| T_{\tilde{x}} \mathcal{F}} \in \R^{+*} SO_2( T_{\tilde{x}} \mathcal{F} )\,.
\end{equation}
\medskip

Now we state the required conditions for the archimedean places of $K$. Observe that 
if $| \, |$ is such a place then all the other archimedean places of $K$ are of the 
form $ z \rightarrow | h(z) |$ where $h$ runs over $G$. Therefore either they are all 
real or all complex.
If they are all real, then $r_1=n$ and the group $G$ acts freely and transitively on the set $S_\infty$
 of archimedean places of $K$. We then require that the action of $G$ on $S_\infty$ coincides 
 with the one of $G$ on the set of corresponding fixed points $\{x_{1,\infty},\, \ldots , x_{n,\infty} \}$.

 If the archimedean places are all complex, then $n= 2 r_2$.  We require that the transitive action of $G$ on $S_\infty$ coincides with the (transitive) action of $G$ on the set of corresponding fixed points $\{z_{1,\infty},\,  \ldots , z_{r_{2},\infty} \}$.  
 For each $j \in \{1,\ldots , r_{2}\}$ there exists a unique element $h_j \in G\setminus \{1\}$  such that 
 $h_j(z_{j,\infty} ) = z_{j,\infty} $. The $h_j$ are all conjugate to each other and satisfy $h_j^2=1$, where $j \in \{1,\ldots , r_{2}\}$.
 In some sense each $h_j$ represents a non canonical model of the complex conjugation.
  Moreover we assume that for any 
 $j \in \{1,\ldots , r_{2}\}$
 \begin{equation} \label{eq:fixes}
 D h_j( z_{j,\infty} )_{| T_{z_{j,\infty}} \mathcal{F}}\, =\, Id,\;\;  D h_j( z_{j,\infty} )_{| \frac{T_{z_{j,\infty}} S_K }{T_{z_{j,\infty}} \mathcal{F}}}\, = \,-Id 
 \end{equation}
\medskip
\subsection{Explanation of the   incompatibility at the archimedean place of $\Q$ between the 
Guillemin-Sternberg trace formula and the explicit formula \eqref{eq:EF}. } $\,$

\bigskip
A priori, Proposition \ref{prop:not}. 2] seems to raise an "objection" in Deninger's approach.
We are going to explain that  the 
Guillemin-Sternberg trace formula and the explicit formula \eqref{eq:EF} {\it are actually compatible. }
Proposition \ref{prop:not}. 2]  is due to the fact 
that $S_\Q$ has a mild singularity at $x_\infty$, whereas  $S_{\Q[i ]}$ is "smooth" at $z_\infty$.

We apply  the previous subsection with $K= \Q[i]$. Thus we require the existence of a degree two ramified covering $\pi_{\Q[i]}: S_{\Q[i ]} \rightarrow S_{\Q}$
with structural group $G=\{ Id, h_1\}$ such that $h_1(z_\infty)= z_\infty$, 
$D h_1(z_\infty)_{| {\rm T}_{z_\infty} \mathcal{F}} = Id_{ { {\rm T}_{z_\infty} \mathcal{F}}}$,
$D h_1 (z_\infty)$ induces $-Id$ on $T_{z_\infty} S_{\Q[i ]}/T_{z_\infty}\mathcal{F}$ and 
$h_1 \circ \phi^t=\phi^t \circ h_1$ for any real $t$.

Then we have (at least formally) the following equality between  distributions of the real variable $t$:

$$
\sum_{j=0}^2 (-1)^j Tr\Big(  (\phi^t)^*\, ; \, \Gamma (S_\Q\, ; \,  \wedge^j T^* \mathcal{F}) \Big) =
$$
$$
\sum_{j=0}^2 (-1)^j Tr\Big( \frac{  (\phi^t)^* + h_1^* (\phi^t)^* }{2} \, ; \, \Gamma (S_{\Q[i]}\, ; \,  \wedge^j T^* \mathcal{F}) \Big) \, = B(t)\,.
$$ We observe that $z_\infty$ is also a fixed point for  $ h_1 \circ \phi^t= \phi^t \circ h_1$, $t\in \R$.
\begin{proposition} \label{prop:OK}

 The contribution of $z_\infty$ in the Guillemin-Sternberg trace formula 
for the term $B(t)$ above is equal to:
$$
\frac{1}{1-e^{-2t}} \; \;{\rm if}\, t>0\,, 
$$ and to
$$
\frac{e^t}{1-e^{2t}}\;\; {\rm if}\, t<0\,.
$$ This contribution matches perfectly with the contribution of the real archimedean place 
in the explicit formula  \eqref{eq:EF}.
\end{proposition}
\begin{proof} The axioms of Section 4.3  (recalled in the beginning of this subsection) and the proof of Lemma \ref{lem:cont} allow to check formally that the
 contribution of $z_\infty$ (for $t\in \R \setminus\{0\}$) in 
 $$ \sum_{j=0}^2 (-1)^j Tr\Big(  h_1^* (\phi^t)^*  \, ; \, \Gamma (S_{\Q[i]}\, ; \,  \wedge^j T^* \mathcal{F}) \Big)
 $$ is equal to:
$$
 \frac{1}{| det ( 1 -D_y (h_1\circ \phi^t(z_\infty)) \, ; \, T_{z_\infty}  S_K/T_{z_\infty} \mathcal{F}) | } = \frac{1}{1+e^{-t}}\, .
$$ 
Now we can
 compute the contribution  of the fixed point $x_\infty \in S_\Q$ in $B(t)$.

For $t>0$ we find:
$$
\frac{1}{2} \Big( \frac{1}{1-e^{-t}} + \frac{1}{1+e^{-t}}\Big) = \frac{1}{1-e^{-2t}}\,.
$$
For $t<0$ we find:
$$
\frac{1}{2} \Big( \frac{e^t}{1-e^{t}} + \frac{e^t}{1+e^{t}}\Big) = \frac{e^t}{1-e^{2t}}\,.
$$  The proposition is proved.
\end{proof}\medskip

\subsection{ Primitive Dirichlet characters and leafwise flat ramified lines bundles.}$\;$
\medskip

Let $\chi$ be a primitive Dirichlet character mod $m$ ($\geq 3$), so $\chi$ is defined by a group homomorphism:
$$
\chi : G = ({ \Z /m \Z} )^* \rightarrow S^1\,.
$$ Consider  the cyclotomic field $K= \Q [ e^{\frac {2 i \pi}{m}}]$, it is a Galois extension of $\Q$  whose Galois group 
is equal to $G$ and has cardinal $\phi(m)$ ($\phi$ being the Euler function). Then define an action of $G$ on $S_K \times \C$ by $h\cdot (z, \lambda)= ( h\cdot z , \chi(h)^{-1} \lambda )$
for any $(h, z, \lambda ) \in G\times S_K \times \C$.

Using the axioms of Section 4.1, we are going to argue that the following leafwise ramified flat line bundle over $S_\Q$:
$$
\mathcal{L}_\chi = \frac{S_K \times \C}{G} \rightarrow S_\Q
$$ should provide the relevant cohomology allowing to interpret the explicit formula $\eqref{eq:EFchi}$ as a Lefschetz trace formula. We set for $j\in \{0,1,2\}$:
$$
\overline{H}^j(\mathcal{L}_\chi)\,=\, \frac{1}{| G|} \sum_{h\in G} h^* ( \overline{H}^j_{\mathcal{F}, K} \otimes_\R \C)= ( \overline{H}^j_{\mathcal{F}, K} \otimes_\R \C)^G\,.
$$
Since $\chi$ is a non trivial (primitive)  character, $G$ acts trivially on $\overline{H}^0_{\mathcal{F}, K}\simeq \R$, $\overline{H}^2_{\mathcal{F}, K}\simeq \R [\lambda_g]$, and 
 we get that $\overline{H}^j(\mathcal{L}_\chi)\,=\,0$ for $j=0$ or $j=2$. Recall that the actions of $(\phi^t)_{t\in \R}$ and $G$ 
 on $S_K$ commute, so the flow $\phi^t$ induces an action (still) denoted 
$( \phi^t)^*$ on $ \overline{H}^1(\mathcal{L}_\chi) $.
Now we are going to give a formal proof an Atiyah-Bott-Lefschetz trace formula whose geometric side 
coincides with the one of the explicit formula  \eqref{eq:EFchi} associated to $\Lambda ( \chi , s)$.
\begin{theorem} \label{thm:Dir} ("informal" theorem) Consider $\alpha \in C^\infty_{compact} (\R^{+*})$. 
 We then have:
\begin{equation} \label{eq:spectral} \begin{aligned}[rcl] &- TR \,\bigl( \int_\R \alpha(t) ( \phi^t)^* d t: \overline{H}^1(\mathcal{L}_\chi) \rightarrow \overline{H}^1(\mathcal{L}_\chi)\, \bigr) = \\ &
\sum_{p \in \mathbb{P} , \,p\wedge m =1} \log p \left( \sum_{k \geq 1} \chi(p)^k
\alpha(k \log p)  \right) + \int_{0}^{+\infty} \frac{\alpha (x) e^{-q x} }{1- e^{- 2 x}}\, d x\, ,
\end{aligned}
\end{equation} where $q\in \{0,1\}$ is such that $\chi(-1)= (-1)^q$.

\end{theorem} 
\begin{proof} We  proceed as in 
the proof of Theorem  \ref{thm:traceram} and
write the left handside of \eqref{eq:spectral} as:
\begin{equation} \label{eq:equiv}
-\frac{1}{| G|} \sum_{h\in G}\, TR \, \int_\R \alpha(s) \, \pi^1_\tau \,(h^{-1}\circ \phi^s)^*  \,d s : \overline{H}^1_{\mathcal{F}, K} \otimes_\R \C \rightarrow \overline{H}^1_{\mathcal{F}, K} \otimes_\R \C \,,
\end{equation}

First we compute (formally) the contributions of the closed orbits according to Guillemin-Sternberg trace formula. Let $p$ be a prime number such that $p\wedge m =1$. Then $p$ is unramified 
in $K= \Q[ e^{\frac{2 i \pi}{m}}]$ ($e=1$) and, with the notations of \eqref{eq:Fr}, 
the residue class $[p] \in (\Z / m \Z)^*$ belongs to $D_{\mathcal{P}_j}$ and is such that $\Theta_j [p]= {\rm Fr}$. See 
\cite{Samuel}[Page 109]. Notice that since here $G$ is abelian, the $D_{\mathcal{P}_j}$ are all equal for $1\leq j \leq r$. We then consider the closed orbit $k\gamma_p$ for $k\in \N^*$, $\gamma_p$ being iterated $k$ times.
Pick up a point $x\in \gamma_p$, for each $j \in \{1, \ldots, r\}$ we select a point 
$\tilde{x}_j \in \gamma_{\mathcal{P}_j}$ such that $\pi_K(\tilde{x}_j) = x$. 
Then using \eqref{eq:Fr} and  \eqref{eq:galoiscov} one immediately gets:
$$
\forall (j, t) \in \{1, \ldots, r\} \times \R , \; [p]^{-k} \phi^t(\tilde{x}_j) = (\tilde{x}_j +t - k \log p )\,.
$$ Now we recall  Assumption 6 in Section 4.1. 
Then,  proceeding as in the proof of Theorem \ref{thm:traceram} and using the equality $ r f = | G |$ ($ = \phi(m)$) 
one checks easily that the contribution of $k\gamma_p$ to the expression \eqref{eq:equiv} is 
equal to
$$
\log p \,   \chi(p)^k\, \alpha( k \log p) \, \epsilon_{k \gamma_p}\,.
$$ But thanks to Assumption 3] a) in Sect. 4.1 (or \eqref{eq:strength}), the sign $\epsilon_{k \gamma_p}$ is equal to $1$ (the determinant 
of a direct similitude being positive).  Therefore the contribution of $k \gamma_p$ is the one expected 
in \eqref{eq:spectral}.

In order to deal with the ramified closed orbits, we use the following:
\begin{lemma} Assume that the prime number $p$ divides $m$ 
so that $p$ is ramified in $K$ and the inertia groups $I_{\mathcal{P}_j}$ (introduced near \eqref{eq:Fr}) are not trivial.
Then the restriction of $\chi$ to any of the $I_{\mathcal{P}_j}$ ($1 \leq j \leq r$) is not trivial.
\end{lemma}
\begin{proof} We follow \cite{Neukirch}[Prop 10.3, Page 61]. Let $m= \prod_l l^{\nu_l}$ be  the prime factorization of $m$ and let 
$f_p$ be smallest positive integer such that 
$$
p^{f_p} \equiv 1 \, {\rm mod}\, ( m/p^{\nu_p} ) \,.
$$ Then one has in $K=\Q[e^{\frac{2 i \pi}{m}}] $ the factorization:
$$
p\,\mathcal{O}_K = ( \mathcal{P}_1 \cdots \mathcal{P}_r)^{\phi(p^{\nu_p})}\, ,
$$ where $\mathcal{P}_1 , \ldots  , \mathcal{P}_r$ are distinct prime ideals, all of norm $p^{f_p}$.
Using the Chinese remainder isomorphism:
$$
\bigl( \frac{\Z }{ m \Z}\bigr)^* \simeq (\frac{\Z}{\frac{m }{p^{\nu_p}} \Z} )^* \times (\frac{\Z }{p^{\nu_p} \Z} )^*\,,
$$ an inspection of the proof of \cite{Neukirch}[Prop 10.3, Page 61] allows to see  that for any $j\in \{1, \ldots , r\}$
$$
D_{\mathcal{P}_j} = D_{\mathcal{P}_1} \simeq < p > \times (\frac{\Z }{p^{\nu_p} \Z} )^*\; , \, 
I_{\mathcal{P}_j} = I_{\mathcal{P}_1} \simeq \{1\} \times (\frac{\Z }{p^{\nu_p} \Z} )^*\,.
$$ Now, the fact that $\chi$ is primitive implies clearly that the restriction of 
$\chi$ to $I_{\mathcal{P}_j}$ is not trivial (use \eqref{eq:prim} with $m'=p^{\frac{m}{\nu_p}}$).
\end{proof}
Suppose now that the prime number $p$ divides $m$ so that $\gamma_p$ is a ramified 
closed orbit. Observe that Condition \eqref{eq:strength} implies that the analogue of \eqref{eq:signcte} is satisfied with 
all the signs being positive.
The previous Lemma, Assumption 6 of Section 4.1 and the proof of Theorem \ref{thm:traceram}  then show formally that the geometric contribution of the ramified $\gamma_p$ 
in \eqref{eq:equiv} is  zero as expected in \eqref{eq:spectral}.

\medskip

Consider now the archimedean places. Assume first that they are all complex, 
then $\phi(m)= 2 r_2$.  Recall the associated fixed points $z_{1,\infty},\, ,\ldots, z_{r_{2},\infty} $ of $\phi^t$ in Section 4.1 and the elements $h_1,\ldots , h_{r_2}$ of the end of Section 4.3; they satisfy $h_j(z_{j,\infty})= z_{j,\infty} ,\, ( 1 \leq j \leq r_2)$. Since $G$ is abelian, the elements $h_j\,  ( 1 \leq j \leq r_2)$ all equal and actually they are equal to   $-1$ (see \cite{Samuel}[Page 109]). Then 
for any $h\in G\setminus \{ 1, -1\}$, $h(z_{j,\infty}) \not=z_{j,\infty}$ ($1\leq j \leq r_2$). The $z_{1,\infty},\, ,\ldots, z_{r_{2},\infty} $ are fixed points 
of $h_1 \phi^t$, $\phi^t$ for all $t\in \R$.
So we are reduced to analyze 
the contributions of  $z_{1,\infty},\, \ldots , z_{r_{2},\infty} $ to:
\begin{equation} \label{eq:redu}
-\frac{1}{| G|} \sum_{h\in \{1, -1\} }\, TR \, \int_\R \alpha(s) \, \pi^1_\tau \,(h^{-1}\circ \phi^s)^*  \,d s : \overline{H}^1_{\mathcal{F}, K} \otimes_\R \C \rightarrow \overline{H}^1_{\mathcal{F}, K} \otimes_\R \C \, .
\end{equation}
We recall the axioms of Section 4.1 and 4.3 (for $h_1= -1$). Then, using the arguments of  the proof of Proposition \ref{prop:OK} and  the Guillemin-Sternberg trace formula, one shows formally that 
the geometric contributions (for $t>0$) of the fixed points of $\phi^t$ 
 in \eqref{eq:equiv} (or \eqref{eq:redu}) is  given by:
$$
\int_{0}^{+\infty} \, \frac{1}{2 r_2} \sum_{j=1}^{r_2} \bigl( \frac{1}{1- e^{-t}} + \frac{\chi(-1)}{1+ e^{-t}} \bigr)\, \alpha(t) d t\, .
$$ But this is exactly the contribution of the archimedean place of $S_\Q$ in \eqref{eq:spectral}. 

Assume now that the archimedean places are all real so that $r_1= \phi(m)$. Then $ e^{\frac {2 i \pi}{m}}$ is real which 
implies that $m=2$. But this case is excluded by assumption.
\end{proof}

\m We have reproved incidentally the fact that a primitive Dirichlet $L-$function is a special case of an Artin $L-$function. 
The general  Artin $L-$functions will be the topic of the next Section.

\bigskip

\section{Artin conjecture as a consequence of an hypothetic Atiyah-Bott-Lefschetz 
proof of the explicit formula for $\widehat{\zeta}_K$.} $\;$

\m
 \m
 
 The following Section  should be viewed as a working programme or a motivation 
 for developing interesting mathematics. We shall perform computations using the various Assumptions 
 of Section 4.
 But, notice that we shall not use here Assumption 5] of Section 4.1 (the one which 
 would imply the Riemann Hypothesis).

\subsection{The Artin $L-$function $\Lambda (K,\chi,s)$.} $\;$

\m

Let $K$ be a finite Galois extension of $\Q$ with Galois group $G$. 
Consider a complex representation $\rho: G \rightarrow GL (V)$ 
where $V$ is a complex vector space of dimension $N$. Its character $\chi : G \rightarrow GL (V)$ is defined by $\chi(h)= {\rm Tr} \,\rho (h),\; h \in G$.  
We are going to recall the definition of the Artin 
$L-$function $\Lambda (K,\chi,s)$ associated to $\rho$ ($\rho$ is determined up to 
isomorphism by $\chi$).

Let $p\in \mathbb{P}$ be a prime number, we use the notations of Section 4.3 and \eqref{eq:Fr}. 
Let $\mathcal{P}_j, \mathcal{P}$ be two prime ideals of $\mathcal{O}_K$ lying over $p$, 
there exists $h\in G$ such that $h \mathcal{P}_j= \mathcal{P}$.
 Choose $\Phi_{\mathcal{P}_j} \in D_{\mathcal{P}_j}$ (modulo $I_{\mathcal{P}_j}$) and $\Phi_{\mathcal{P}} \in D_\mathcal{P}$
 such that $\Theta_j(\Phi_{\mathcal{P}_j}) = {\rm Fr}$ and $\Theta (\Phi_{\mathcal{P}}) = {\rm Fr}$. 
 Denote now by $V^{I_{\mathcal{P}_j}}$ the subset  of vectors $V$ which are invariant under the action of $I_{\mathcal{P}_j}$.
 It is then clear 
 that ${\rm det}\, ( Id - p^{-s} \rho( \Phi_{\mathcal{P}_j}) ; V^{I_{\mathcal{P}_j}})$ and 
 ${\rm det}\, ( Id - p^{-s} \rho(\Phi_{\mathcal{P}}) ; V^{I_{\mathcal{P}}})$ do not depend on the choice of 
 respectively $ \Phi_{\mathcal{P}_j}$, $\Phi_\mathcal{P}$. Therefore, we can assume that $\Phi_\mathcal{P}= h \Phi_{\mathcal{P}_j }h^{-1}$. Then since, $h D_{\mathcal{P}_j }h^{-1}= D_\mathcal{P}$ and $h I_{\mathcal{P}_j }h^{-1}= I_\mathcal{P}$, 
 we obtain:
 $$
 \forall s\in \C,\; {\rm det}\, ( Id - p^{-s} \rho(\Phi_{\mathcal{P}_j}) ; V^{I_{\mathcal{P}_j}}) = {\rm det}\, ( Id - p^{-s} \rho(\Phi_{\mathcal{P}}) ; V^{I_{\mathcal{P}}})\,.
 $$
 
If all the archimedean absolute values of $K$ are real we set $n^+_\sigma= N = {\rm dim}\, V$, $n^-_\sigma=0$. If all the archimedean absolute values of $K$ are complex, we set:
 $$
 n^+_\sigma= {\rm dim}\, \ker ( \rho(h_j) - Id ),\; n^-_\sigma= {\rm dim}\, \ker ( \rho(h_j) + Id )\,,
 $$ where $h_j$ is any of the elements $h_1, \ldots , h_{r_2}$ of $G$ introduced at the end of Section 4.3 (they are all conjugate to each other and satisfy $h_j^2=1$).
 
 Now, for $s\in \C$  such that $\Re s >1$, we define the Artin $L-$function as:
 \begin{equation} \label{def:A}
 \Lambda (K,\chi,s) = ({\mathcal{N}}( K,\chi))^{s/2}  \Gamma_\R(s)^{n_\sigma^+} \Gamma_\R(s+1)^{n^-_\sigma} \,
  \prod_{p \in \mathbb{P}} \frac{1}{{\rm det}\, ( Id - p^{-s} \rho(\Phi_{\mathcal{P}}) ; V^{I_{\mathcal{P}}})}\,,
 \end{equation} where the positive integer ${\mathcal{N}}( K,\chi)$ denotes the norm of the Artin conductor of $\chi$.
 Recall that $\Gamma_\R(s)= \pi^{-s/2} \Gamma(s/2)$.
 
 Now we  recall Brauer's theorem in order to explain  the meromorphic continuation of $ \Lambda (K,\chi,s)$.
 The character $\chi$ is an integral linear combination $\chi= \sum_{i=1}^k n_i \chi_{i *}$ where the $n_i \in \Z$ and  the $\chi_{i *} $ 
 are induced from characters $\chi_i$ of degree $1$ on subgroups $H_i= Gal (K : L_i)$, $L_i$ a suitable subfield of $K$.
From this, one can deduce that:
\begin{equation} \label{eq:Arti}
\Lambda (K,\chi,s) = \prod_{i=1}^k \Lambda (K,\chi_{i *},s)^{n_i} = \prod_{i=1}^k \Lambda( \widetilde{\chi_i} , s)^{n_i}\, , 
\end{equation} where $ \Lambda( \widetilde{\chi_i} , s)$ is the $L-$function attached to the Groessencharakter 
$\widetilde{\chi_i}$ associated to $\chi_i$. From these identities, one deduces that $\Lambda (K,\chi,s)$ admits 
a meromorphic continuation to $\C$ with zeroes and poles all belonging to the critical strip $0\leq \Re s \leq 1$. 
Moreover, it satisfies the functional equation:
$$
\Lambda (K,\chi,s) = W(\chi)  \Lambda (K,\overline{\chi},1-s)\,,
$$ where $W(\chi)$ is a complex constant of modulus $1$.

\m
\noindent {\bf Artin Conjecture}. If the representation $\rho$ is irreducible then $\Lambda (K,\chi,s) $ 
is entire (without any pole).

\m

So Artin conjecture means that each  zero of $\Lambda( \widetilde{\chi_i} , s)$ for negative $n_i$ in \eqref{eq:Arti}
is  compensated by a zero of another $\Lambda( \widetilde{\chi_j} , s)$ for positive $n_j$.
This conjecture is proved when $G$ is abelian. In the non abelian case, it is proved in several particular cases using deep methods (see Taylor's survey 
\cite{Taylor}), but it remains widely open in the general case. 

\m
\subsection{Explicit Formulas and Trace Formulas.}$\;$

\begin{theorem} Let $\{ \lambda_k, k \in I\}$  (resp. $\{ \mu_j, j \in J\}$ ) be the set of zeroes (resp. poles) 
of $\Lambda (K,\chi,s)$. Let $\alpha \in C^\infty_{compact} ( \R^{+*})$. Then one has:
\begin{equation} \label{eq:standard}
\begin{aligned}[rcl]
& - \sum_{k \in I} \int_0^{+\infty} \alpha(s) e^{ s \lambda_k}\, d s \, + \sum_{j \in J} \int_0^{+\infty} \alpha(s) e^{ s  \mu_j}\, d s \, = \\ &
\sum_{p \in \mathbb{P}} \log p \sum_{ n \geq 1} \alpha(n \log p) {\rm Tr}\,(  \Phi^n_\mathcal{P}: V^{I_\mathcal{P}} ) +
\int_0^{+\infty} \frac{\alpha(x)}{1 - e^{-2x}} (n^+_\sigma + n^-_\sigma e^{-x}) \, d x\,.
\end{aligned}
\end{equation} 
\end{theorem}
\begin{proof} One proceeds exactly as for the proof of the explicit formula \eqref{eq:EFchi}, using  the functional equation for 
$\Lambda (K,\chi,s)$ and its Eulerian product.
\end{proof}
\bigskip

Now we define an action of $G$ on $S_K \times V$ by
$$
h\cdot (z, v) = (h \cdot z , \rho^{-1}(h)\cdot v),\; \forall (h,z,v) \in G \times S_K \times V\,.
$$ Then using the axioms of Section 4.1, we are going to argue that the following leafwise 
ramified flat vector bundle over $S_\Q$:
$$
\cE_\rho = \frac{S_K \times V}{G} \rightarrow S_\Q
$$ should provide the relevant cohomology allowing to exhibit an explicit formula for $\Lambda (K,\chi,s) $ 
via a Lefschetz trace formula. 
We set for $j\in \{0,1,2\}$:
$$
\overline{H}^j(\mathcal{E}_\rho)\,=\, \frac{1}{| G|} \sum_{h\in G} h^* ( \overline{H}^j_{\mathcal{F}, K} \otimes_\R V)= ( \overline{H}^j_{\mathcal{F}, K} \otimes_\R V)^G\,.
$$ Recall that the actions of $(\phi^t)$ and $G$ on $S_K$ are assumed to commute.
We shall still denote by $(\phi^t)^*$ the action on $\overline{H}^j(\mathcal{E}_\rho)$ induced by the flow $(\phi^t)$.

Next we recall and use Assumption 6] of Section 4.1. According to the hypothesis of Section 4.3, 
for each zero $z_q$ of $\widehat{\zeta}_K$, $G$ leaves each 
$$
\ker (\theta_1 - z_q Id)^{n(z_q)} \otimes_\C V = W_q
$$ globally invariant and commutes with $\theta_1 \otimes_\C Id_V$ (which we shall denote simply 
$\theta_1$). Next, observe that $\theta_1 -z_q Id$ induces a nilpotent endomorphism of 
$\frac{1}{| G |} \sum_{h \in G} h \cdot W_q$, therefore:
\begin{equation} \label{eq:zero}
\forall t \in \R,\; {\rm Tr}\, ( e^{t \theta_1}:\, \frac{1}{| G |} \sum_{h \in G} h \cdot W_q\,  ) \, = \, d_q \,e^{t z_q}\,,
\end{equation} where $d_q=$ dim $\frac{1}{| G |}\sum_{h \in G} h \cdot W_q$. Notice that a priori some of the $d_q$ 
may be equal to $0$.
\begin{theorem} ("Informal" Theorem).  Assume that the representation $\rho$ is irreducible.
Let $\alpha \in C^\infty_{compact} ( \R^{+*})$. Then one has:
\begin{equation} \label{eq:Artin}
\begin{aligned}[rcl]
&-\sum_{z_q \in 
\widehat{\zeta}_K^{-1}\{0\}}\, d_q  \int_0^{+\infty} \alpha(s) e^{ s z_q}\, d s \, 
= \\ &
\sum_{p \in \mathbb{P}} \log p \sum_{ k \geq 1} \alpha(k \log p) {\rm Tr}\,(  \Phi^k_\mathcal{P}: V^{I_\mathcal{P}} ) +
 \int_0^{+\infty} \frac{\alpha(x)}{1 - e^{-2x}} (n^+_\sigma + n^-_\sigma e^{-x}) \, d x\,.
\end{aligned}
\end{equation} 
\end{theorem}
\noindent {\bf Comment.} Actually, we can only  assume that $V^G = \{0\}$, which is of course weaker than $\rho$ irreducible.
The $\Gamma$ factor $\Gamma_\R(s)^{n_\sigma^+} \Gamma_\R(s+1)^{n^-_\sigma}$ comes in the definition of 
$\Lambda (K, \chi, s)$ as a parachute and is not well motivated by a mathematical structure. It was introduced there (by Artin himself?)
in order to get \eqref{eq:Arti}.
That is why 
in the spectral side of \eqref{eq:standard} we have "uncontroled poles". This $\Gamma$ factor appears 
naturally in the computation of the contribution of the fixed points in \eqref{eq:Artin} and that is why the spectral side of \eqref{eq:Artin} is better "controled".

\begin{proof} We shall use the arguments of the proofs of Theorems  \ref{thm:traceram} and \ref{thm:Dir}. 
Since $G$ acts trivially on  $\overline{H}^0_{\mathcal{F}, K} \simeq \R$ and $\overline{H}^2_{\mathcal{F}, K} \simeq \R [\lambda_g]$, one gets that $\overline{H}^j(\mathcal{E}_\rho) =\{0\}$ for $j=0,2$.
Therefore, using \eqref{eq:zero}, one gets: 
$$
\sum_{j=0}^2 (-1)^j TR \,\bigl( \int_\R \alpha(t) ( \phi^t)^* d t: \overline{H}^j(\mathcal{E}_\rho) \, \bigr) =-\sum_{z_q \in 
\widehat{\zeta}_K^{-1}\{0\}}\, d_q  \int_0^{+\infty} \alpha(s) e^{ s z_q}\, d s \,.
$$ We now write the left hand side of this equality under the form:
\begin{equation} \label{eq:aba}
\begin{aligned}[rcl]
& \frac{1}{| G|} \sum_{h\in G}\, \sum_{j=0}^2 (-1)^j\,TR \, \bigl( \int_\R \alpha(s) \, \pi^j_\tau \,(h^{-1}\circ \phi^s)^*  \,d s : \overline{H}^j_{\mathcal{F}, K} \otimes_\R V\bigr) \, =\, \\ &
-\frac{1}{| G|} \sum_{h\in G}\,TR \, \bigl( \int_\R \alpha(s) \, \pi^1_\tau \,(h^{-1}\circ \phi^s)^*  \,d s : \overline{H}^1_{\mathcal{F}, K} \otimes_\R V\bigr) 
\end{aligned}
\end{equation} where $\pi^j_\tau$ (written for $\pi^j_\tau \otimes Id_V$)  denotes the Hodge projection onto leafwise harmonic forms. We now proceed to compute formally the 
geometric contributions in \eqref{eq:aba} of the closed orbits and of the fixed points of $(\phi^t)$ according to 
the Guillemin-Sternberg trace formula.

Thus,  we consider a prime number $p\in \mathbb{P}$ and the contribution of the closed orbit $k \gamma_p$, $k\in \N^*$. Let $\mathcal{P} $ 
be any of the prime ideals $\mathcal{P}_1, \ldots , \mathcal{P}_r$ of $\mathcal{O}_K$ such that $p \mathcal{O}_K= \mathcal{P}_1^e \ldots  \mathcal{P}_r^e$. Fix $x_0 \in \gamma_p$ and $\tilde{x}_0\in \gamma_\mathcal{P}$ such that $\pi_K(\tilde{x}_0) = x_0$. The axioms of Section 4.3 and \eqref{eq:galoiscov}  imply that:
$$
\forall t \in \R,\; \Phi_\mathcal{P}^{-k} \phi^t(\tilde{x}_0)= (\tilde{x}_0 + t - k\log p)\,.
$$

First assume that $p$ is not ramified in $K$ ($e=1$). Then, there are exactly $| G |$ curves (of the flow) in $S_K$ lying (via $\pi_K$) over $k \gamma_p$. 
They are given by
$$
t \rightarrow l\cdot \Phi_\mathcal{P}^{-k} l^{-1}\cdot\phi^t(l\cdot \tilde{x}_0), \; 0\leq t \leq k \log p,\; l \in G\,.
$$ Then the proof of Theorem  \ref{thm:AK} (\cite{A-K2}) shows that the geometric contribution of $k \gamma_p$ to \eqref{eq:aba} is computed according to Proposition \ref{prop:G-S} and is equal to:
$$
\frac{\log p}{| G|} \sum_{l \in G} \sum_{j=0}^2 (-1)^j \frac{ Tr\, \bigl( ^{t}D (l \Phi_\mathcal{P}^{-k} l^{-1}\circ \phi^{k\log p} ) ( l\cdot \tilde{x_0})\, ;\, \wedge^j T^*_{l\cdot \tilde{x_0}} \mathcal{F}  \bigr)}
{| \,{\rm  det}\, \bigl( {\rm id} - D (l \Phi_\mathcal{P}^{-k} l^{-1}\circ \phi^{k \log p})\bigr)_{| T_{l\cdot \tilde{x_0}} \mathcal{F}}\, |}\,\,\chi(l \Phi_\mathcal{P}^{k} l^{-1}) \, \alpha( k \log p)\, .
$$ The sign assumption \eqref{eq:strength} shows that all these determinants are positive. Therefore, the geometric contribution of $k \gamma_p$ to \eqref{eq:aba} is equal to $\log p\, \chi ( \Phi_\mathcal{P}^{k})\, \alpha (k \log p)$ as expected in \eqref{eq:Artin}.

\smallskip
Next assume that $p$ is  ramified in $K$ ($e>1$).
So $I_\mathcal{P}=G_{\tilde{x_0}}= \{u \in G /\, u\cdot \tilde{x_0}= \tilde{x_0}\}$ 
is not trivial. Then there are exactly $| G/ G_{\tilde{x_0}}|$ curves (of the flow) upstairs on $\tilde{X}$ which correspond 
(via $\pi$) to $k \gamma_{\log p}$. 
They are given by $t \mapsto l_j \Phi_\mathcal{P}^{-k} l_j^{-1}\circ \phi^t ( l_j \cdot \tilde{x}_0 )$ 
( $ 0 \leq t \leq k \log p$), where the $l_j$ ($1\leq j \leq r f$) run over a system of representatives of cosets of 
$ G/ G_{\tilde{x_0}}$. We have to count
each such curve $| G_{\tilde{x_0}}|$ times but with (possibly) 
different monodromies (ie action on the vector space factor $V$). More precisely,  for each representative $l_j$ the monodromy of  the curve labeled
\begin{equation} \label{eq:ramiA}
t \mapsto   l_j \Phi_\mathcal{P}^{-k} l_j^{-1} u\circ \phi^t ( l_j \cdot \tilde{x}_0 ),\; {\rm with}\, u \in l_j G_{\tilde{x_0}} l_j^{-1}\, ,
\end{equation} is equal to $\rho(u^{-1} l_j \Phi_\mathcal{P}^{k} l_j^{-1} )$. The proof given above in the unramified case, with a more precise use of the sign assumption \eqref{eq:strength}, allows to check that the geometric contribution of $k \gamma_p$ to \eqref{eq:aba} is equal to 
$$
\frac{\log p \, \alpha( k \log p)\, }{| G |} \sum_{j=1}^{f r} \sum_{ u' \in I_\mathcal{P}} Tr\, ( \rho ( l_j \Phi^k_{\mathcal{P}} l_j^{-1} l_j u' l_j^{-1} )\,=\,
\log p \, \alpha( k \log p) \, Tr \bigl( \rho(\Phi^k_{\mathcal{P}} )\,\frac{1}{e} \sum_{ u' \in I_\mathcal{P}} \rho(u') \bigr) \,.
$$ Since $ \frac{1}{e}\sum_{ u' \in I_\mathcal{P}} \rho(u')$  is a projection of $V$ onto $V^{I_\mathcal{P}}$ which commutes with $\rho(\Phi^k_\mathcal{P})$, this contribution coincides 
with 
$$\log p \, \alpha(k \log p) \, {\rm Tr}\,(  \rho( \Phi^k_\mathcal{P}) : V^{I_\mathcal{P}} )  $$ as expected in \eqref{eq:Artin}.

Now we come to the contribution to \eqref{eq:aba} of the fixed points (i.e. the archimedean places of $K$). We proceed as in the proof of Theorem \ref{thm:Dir} and give only a sketch. 

First assume that all the archimedean places are complex. 
Then $| G |= 2 r_2$, recall the associated fixed points $z_{1,\infty},\, ,\ldots, z_{r_{2},\infty} $ of $\phi^t$ in Section 4.1 and  the associated elements $h_j$  of $G$ in Section 4.3.  Then for each $j \in \{1, \ldots , r_2\}$,  $z_{j,\infty} $ is a  fixed point 
of $h_j \phi^t$, $\phi^t$ for all $t\in \R$, and 
for any $h\in G\setminus \{ 1, h_j\}$, $h(z_{j,\infty}) \not=z_{j,\infty}$. Thus we have to determine the contributions of  $z_{1,\infty},\, \ldots , z_{r_{2},\infty} $ to:
\begin{equation} \label{eq:reduA}
-\frac{1}{| G|} \sum_{h\in \{1, h_1,\ldots , h_{r_2}\} }\, TR \, \bigl( \, \int_\R \alpha(s) \, \pi^1_\tau \,(h^{-1}\circ \phi^s)^*  \,d s : \overline{H}^1_{\mathcal{F}, K} \otimes_\R V \rightarrow \overline{H}^1_{\mathcal{F}, K} \otimes_\R V \,\bigr) \, .
\end{equation}
We recall the axioms of Sections 4.1 and 4.3 (for the $h_j$). Then, using the arguments of  the proof of Proposition \ref{prop:OK} and  the Guillemin-Sternberg trace formula, one shows formally that 
the geometric contributions (for $t>0$) of the fixed points of $\phi^t$ 
 in \eqref{eq:aba} (or \eqref{eq:reduA}) is  given by:
$$
\int_{0}^{+\infty} \, \frac{1}{2 r_2} \sum_{j=1}^{r_2} \bigl( \frac{Tr\, \rho(Id_V)}{1- e^{-t}} + \frac{Tr\, \rho(h_j)}{1+ e^{-t}} \bigr)\, \alpha(t) d t\, .
$$

Recall that $\rho(h_j)$ is a symmetry of $V$
such that $Tr\, \rho(h_j) = n^+_\sigma - n^-_\sigma$. One then computes that for any real $x>0$:
$$
\frac{ Tr\, \rho(Id_V)} { 1-e^{-x}} + \frac{ Tr\, \rho(h_j)} { 1+e^{-x}} = \frac{ n^+_\sigma + n^-_\sigma} { 1-e^{-x}} + \frac{ n^+_\sigma - n^-_\sigma} { 1+e^{-x}} = \frac{ 2 n^+_\sigma}{ 1-e^{-2x}} + \frac{ 2 n^-_\sigma e^{-x}}{ 1-e^{-2x}}\,.
$$ One then proves easily that the contribution to \eqref{eq:aba} of the fixed points coincides with
$$ \int_0^{+\infty} \frac{\alpha(x)}{1 - e^{-2x}} (n^+_\sigma + n^-_\sigma e^{-x}) \, d x$$
 as expected in \eqref{eq:Artin}.

  When all the archimedean places are real, the situation is much simpler.
 Recall the set of fixed points $\{x_{1,\infty},\, \ldots, x_{r_1,\infty} \}$ associated to the real archimedean places, 
then for any $h\in G\setminus \{ 1 \}$,  $h(x_{j,\infty}) \not=x_{j,\infty}$ ($1\leq j \leq r_1$).
So we are reduced to analyze 
the contributions of  $x_{1,\infty},\, \ldots , x_{r_1,\infty}  $ to:
\begin{equation}\label{eq:redu1}
-\frac{1}{| G|}  TR \, \bigl(\,\int_\R \alpha(s) \, \pi^1_\tau \,(\phi^s)^*  \,d s : \overline{H}^1_{\mathcal{F}, K} \otimes_\R V \rightarrow \overline{H}^1_{\mathcal{F}, K} \otimes_\R V \, \bigr)\, .
\end{equation}

Therefore, the axioms of Section 4.1  and the Guillemin-Sternberg trace formula show formally that 
the geometric contributions (for $t>0$) of the fixed points of $\phi^t$ 
in \eqref{eq:aba}  (or \eqref{eq:redu1}) is given by:
$$
\int_{0}^{+\infty} \, \frac{1}{r_1} (\sum_{j=1}^{r_1}  \frac{Tr\, \rho(Id_V)}{1- e^{-2t}} ) \, \alpha(t) d t\,.
$$ But, since here  $ Tr\, \rho(Id_V)\, = N= n^+_\sigma$ and  $n^-_\sigma =0$, this is exactly the expected contribution  in \eqref{eq:Artin}. 
\end{proof}

\m

\begin{theorem} Let $(u_k)_{k\in A}$ and $(v_l)_{l\in B}$ be two sequences of 
points (with possible multiplicity) of the critical strip $\{z\in \C,\; 0\leq \Re z\leq 1\}$, $A$ and $B$ being two subsets of $\N$. Assume that $ \sum_{k\in A} \frac{1}{1 + | u_k|^2} + \sum_{l\in B} \frac{1}{1 + | v_l |^2} <+\infty$ and that for 
any $\alpha \in C^\infty_{compact}(\R^{+*})$:
\begin{equation} \label{eq:alpha}
\sum_{k\in A} \int_0^{+\infty} \alpha(s) e^{s u_k} \, d s \, = \, \sum_{l\in B} \int_0^{+\infty} \alpha(s) e^{s v_l} \, d s\,.
\end{equation} Then, there exists a bijection $\xi: A \rightarrow B$ such that 
for any $k \in A, u_k= v_{\xi(k)}$ with the same multiplicity.
\end{theorem}
\begin{proof} We give only a sketch. First, in \eqref{eq:alpha} we can replace 
$\alpha(s)$ by $\alpha(s) e^{-2 s}$. Then, two successive integration by parts allow to see that:
$$
 \int_0^{+\infty}\alpha^{''}(s) \sum_{k\in A}  \frac{e^{s( u_k-2)}}{( u_k-2)^2} \, d s \, = \,
 \int_0^{+\infty} \alpha^{''}(s) \sum_{l\in B}  \frac{e^{s( v_l-2)}}{( v_l-2)^2} \, d s\,.
$$ Therefore there exists two constants $M_1, M_2$ such that:
$$
\forall s \in [0, +\infty[,\; \sum_{k\in A}  \frac{e^{s( u_k-2)}}{( u_k-2)^2}\,=\, \sum_{l\in B}  \frac{e^{s( v_l-2)}}{( v_l-2)^2}\, 
+ M_1 + s M_2 \, .
$$ Letting $s\rightarrow +\infty$, one gets $M_1=M_2=0$. Now applying inductively $\int_{+\infty}^s$ to the previous 
identity, one obtains for every $r\in \N$:
$$
\sum_{k\in A}  \frac{1}{( u_k-2)^{2+r}}\,=\, \sum_{l\in B}  \frac{1}{( v_l-2)^{2+r}}\,.
$$ In order to finish the proof, we use an elegant argument pointed out to us by Vincent Lafforgue.
The previous equality implies that all the derivatives at $0$ of the following  meromorphic function vanish:
$$
z\mapsto \sum_{k\in A}  \frac{1}{( z+ u_k-2)^{2}}\,-\, \sum_{l\in B}  \frac{1}{( z+v_l-2)^{2}}\,.
$$ Hence this function is identically zero, which proves the theorem.
\end{proof}

The conjunction of the three last Theorems (formally) would imply Artin conjecture and that 
the zeroes of  $\Lambda (K,\chi,s)$ are the zeroes $z_q$ of $ \widehat{\zeta}_K$ with multiciplity $d_q$. Of course it is understood that if $d_q=0$ then $z_q$ is not a zero of $\Lambda (K,\chi,s)$.

	\bibliographystyle{amsalpha}

\begin{thebibliography}{ACDE}

\bibitem[A-K00]{A-K2} J. A. Alvarez Lopez and Y. Kordyukov: {\it Distributional 
Betti numbers of transitive foliations of codimension one}. Foliations: geometry 
and dynamics (Warsaw, 2000), World Sci. Publishing, River Edge, NJ, (2002), 
pages 159-183.

\bibitem[A-K01]{A-K1} J. A. Alvarez Lopez and Y. Kordyukov: {\it Long time 
behaviour of leafwise heat flow for Riemannian foliations}. Compositio Math. 
{\bf 125} (2001), pages 129-153.

\bibitem[A-K05]{A-K3} J. A. Alvarez Lopez and Y. Kordyukov: {\it
Distributional Betti numbers for Lie foliations}. To appear in the proceedings of 
Bedlewo conference 2005 on $C^*-$algebras.

\bibitem[AKL]{AKL} J. Alvarez-Lopez, Y. Kordyukov, and E. Leichtnam: {\it A trace formula for foliated flows } . 
Work in progress.

\bibitem[Cara]{Cara} O. Caramello: {\it Topological Galois Theory}.  arxiv:math.CT/1301.0300, 2012.

\bibitem[Co94]{C 1} A. Connes: {\it \underline{Noncommutative Geometry}}. 
Academic Press.

\bibitem[Co99]{C 2} A. Connes: {\it Trace formula in noncommutative
geometry and the zeroes of the Riemann zeta function}. Selecta Math.
(N.S.) 5 (1999), no. 1, pages 29-106.














\bibitem[De94]{D 3} C. Deninger: {\it Motivic $L-$functions 
and regularized determinants }. Proc. Symp. Pure Math. {\bf 55}, 1 (1994), pages 707-743. 

\bibitem[De94b]{D 3b} C. Deninger: {\it Motivic $L-$functions 
and regularized determinants II}.
F. Catanese (ed.), Arithmetic Geometry, Cortona, 1994
Symp. Math. 37, Cambridge Univ. Press 1997, pages 138-156.

\bibitem[De98]{D 1} C. Deninger: {\it Some analogies between number
theory and dynamical systems onf foliated spaces}.
Doc. Math. J. DMV. Extra volume ICM I, (1998), pages 23-46.

\bibitem[De99]{D 0c} C. Deninger: {\it On dynamical systems and their possible significance 
for Arithmetic Geometry}. In: A. Reznikov, N. Schappacher (eds.), Regulators in Analysis, 
Geometry and Number Theory. Progress in Mathematics {\bf 171}, (1999), 
Birkhauser, pages 29-87. 


\bibitem[De-Si00] {D 1b} C. Deninger and W. Singhof: {\it A note on dynamical trace formulas}, 
Dynamical, spectral and arithmetic zeta functions (San Antonio, TX, 1999), Contemp. 
Math., 290, Amer. Math. Soc., Providence, RI, (2001), pages 41-55.

\bibitem[De-Si02] {D-S} C. Deninger and W. Singhof: {\it Real polarizable Hodge structures 
arising from foliations}.  Annals of Global Analysis and Geometry 21,  2002,  pages 377-399

\bibitem[De01]{D 2b} C. Deninger: {\it Number theory and dynamical systems 
on foliated spaces}. Jahresberichte der DMV. {\bf 103},  (2001), No 3, pages 
79-100.

\bibitem[De01b]{D 1c} C. Deninger: {\it A note on arithmetic topology and dynamical systems}. 
Algebraic number theory and algebraic geometry, Contemp. Math., 300, Amer. Math. Soc., 
Providence, RI, (2002), pages 99-114.



\bibitem[De02]{D 2} C. Deninger: {\it On the nature of explicit
formulas in analytic number theory, a simple example}. Number theoretic methods (Iizuka, 2001),
 Dev. Math., 8, Kluwer Acad. Publ., Dordrecht, (2002), pages 97-118.
 
\bibitem[De07b]{D 8} C. Deninger: {\it Analogies between analysis on foliated spaces 
and arithmetic geometry}, to appear in the Proceedings of a conference 
in the honour of Hermann Weyl.

 \bibitem[F-M]{FM} M. Flach and B. Morin:  {\it On the Weil-Etale topos of regular arithmetic schemes}. 
 Doc. Math. 17 (2012), pages 313-399.

\bibitem[Gelbart]{Gelbart}  S. Gelbart: {\it Class field theory, the Langlands program, and its application to number theory }. Automorphic forms and the Langlands program, Adv. Lect. Math. (ALM), 9, Int. Press, Somerville, MA, 2010, pages 21-67.


\bibitem[G-S77]{G-S} V. Guillemin and S. Sternberg: \underline{{\it Geometric 
Asymptotics}}. Mathematical surveys and monographs. Number 14.
Published by the A.M.S, Providence, R.I. 1977.






\bibitem[L.Lafforgue]{Lafforgue} L. Lafforgue: {\it L'ind\'ependance  la cohomologie l-adique et la correspondance de Langlands sont-elles des \'equivalences de Morita entre topos classifiants ? } Written text that can be found on 
http://www.ihes.fr/~lafforgue/publications.html .

 \bibitem[Lapid]{Lapid} E. Lapid: {\it Introductory notes on the trace formula}. Automorphic forms and the Langlands program, Adv. Lect. Math. (ALM), 9, Int. Press, Somerville, MA, 2010, pages 135-175.
 
\bibitem[Lei03] {El} E. Leichtnam: {\it An invitation to 
Deninger's work on arithmetic zeta functions}.  Geometry, spectral theory, groups, and dynamics, Contemp. Math., 387, Amer. Math. Soc., Providence, RI, (2005), pages 
201-236.

\bibitem[Lei07] {El1} E. Leichtnam: {\it Scaling group flow and 
Lefschetz trace formula for laminated spaces with a $p-$adic transversal}, Bulletin 
des Sciences Mathematiques, 131,  (2007), pages 638-669.

\bibitem[Lei08] {Rend} E. Leichtnam: {\it On the analogy between Arithmetic Geometry and 
foliated spaces}. Rendiconti di Matematica, Serie VII Volume 28, Roma (2008), pages 163-188.





\bibitem[Meyer05]{Meyer} R. Meyer: {\it On a representation of the idele class group related to primes and zeros of L-functions}. Duke Math. J. 127 (2005), no. 3, pages 519-595.

\bibitem[Mor] {Morin} B. Morin: {\it Sur l'analogie entre le système dynamique de Deninger et le topos Weil-Etale. (French) [On the analogy between Deninger's dynamical system and the Weil-Etale topos}.  Tohoku Math. J. (2) 63 (2011), no. 3, pages 329-361. 

\bibitem[Neu99] {Neukirch} J. Neukirch: {\it \underline{Algebraic number theory}}. Springer, 
volume 322. (1999).


 







\bibitem[Sam]{Samuel} P. Samuel: {\it \underline{Th\'eorie alg\'ebrique des nombres}}. Editions Hermann 1967.

\bibitem[Se60]{Serre} J-P. Serre: {\it Analogues Kaehleriens de certaines 
conjectures de Weil}. Annals of Math. 65, (1960), pages 392-394.

\bibitem[Taylor]{Taylor} R. Taylor: {\it Galois representations}. 
Annales de la Facult\'e des Sciences de Toulouse 13 (2004), pages 73-119.

\end{thebibliography}

\end{document}